\documentclass[a4paper,11pt,reqno]{amsart}
\usepackage{amsxtra}
\usepackage{color}
\usepackage{amsopn}
\usepackage{amsmath,amsthm,amssymb,arydshln}
\usepackage{mathrsfs,mathtools}
\usepackage{stmaryrd}
\usepackage{graphicx}
\usepackage{bm}
\usepackage{enumerate}
\usepackage{xcolor}
\usepackage{pifont}
\usepackage{adjustbox}
\usepackage{tikz}
\usepackage{color}
\usepackage{hyperref}

\newtheorem{theorem}{Theorem}[section]
\newtheorem{corollary}[theorem]{Corollary}
\newtheorem{proposition}[theorem]{Proposition}
\newtheorem{lemma}[theorem]{Lemma}
\newtheorem*{theorem*}{Theorem}

\theoremstyle{definition}

\newtheorem{example}[theorem]{Example}

\newtheorem{remark}[theorem]{Remark}

\makeatletter
\@namedef{subjclassname@2020}{%
  \textup{2020} Mathematics Subject Classification}
\makeatother

\newcommand{\bC}{\mathbb{C}}

\newcommand{\bR}{\mathbb{R}}

\newcommand{\beq}{\begin{equation}}
\newcommand{\eeq}{\end{equation}}

\newcommand{\e}{\mathrm{e}}

\newcommand{\f}{\varphi}

\renewcommand{\l}{\lambda}

\newcommand{\s}{\sigma}

\newcommand{\psip}{\psi_{\sst+}}
\newcommand{\psim}{\psi_{\sst-}}

\newcommand{\SU}{{\mathrm{SU}}}
\newcommand{\On}{{\mathrm O}}
\newcommand{\SO}{{\mathrm {SO}}}

\newcommand{\GL}{{\mathrm {GL}}}

\newcommand{\G}{{\mathrm G}}

\newcommand{\W}{\wedge}

\renewcommand{\d}{\mathrm{d}}
\DeclareMathOperator\tr{tr}

\DeclareMathOperator\End{End}

\DeclareMathOperator\ad{ad}

\DeclareMathOperator\vol{vol}

\newcommand{\Scal}{{\rm Scal}}

\newcommand{\ga}{\mathfrak{a}}

\newcommand{\gh}{\mathfrak{h}}

\newcommand{\gn}{\mathfrak{n}}

\newcommand{\gr}{\mathfrak{r}}

\newcommand{\gz}{\mathfrak{z}}

\newcommand{\so}{\mathfrak{so}}

\newcommand{\ma}{\textcolor{magenta}}
\newcommand{\co}{\lrcorner}

\newcommand{\st}{\ |\ }

\newcommand{\diag}{{\rm diag}}
\newcommand{\rank}{\mathrm{rank}}

\newcommand{\sst}{\scriptscriptstyle}

\newcommand{\mn}{\mathfrak n }

\newcommand{\mz}{\mathfrak z }

\newcommand{\mk}{\mathfrak k }
\renewcommand{\ma}{\mathfrak a }

\newcommand{\mh}{\mathfrak h }

\newcommand{\mr}{\mathfrak r }

\newcommand{\lela}{\left\langle}
\newcommand{\rira}{\right\rangle}

\newcommand{\mcB}{\mathcal B}

\renewcommand{\l}{\lambda}

\newcommand{\R}{\mathbb R}

\newcommand{\nc}{\newcommand}
 \nc{\iso}{\mathfrak{iso}}
 \nc{\sso}{\mathfrak{so}}
\nc{\Sym}{\mathrm{Sym}}
 \nc{\pr}{\operatorname{pr}} 
 \nc{\Dera}{\operatorname{Dera}} \nc{\Auto}{\operatorname{Auto}}

\nc{\noi}{\noindent}

\numberwithin{equation}{section}

\title[Purely coclosed G$_2$-structures on 2-step nilpotent Lie groups]{Purely coclosed G$_{\mathbf2}$-structures on 2-step nilpotent Lie groups}

\author{Viviana del Barco}
\address{Universit\'e Paris-Saclay, CNRS, Laboratoire de math\'ematiques d'Orsay, 91405, Orsay, France and Universidad Nacional de Rosario, CONICET, 2000, Rosario, Argentina}
\email{viviana.del-barco@math.u-psud.fr}

\author{Andrei Moroianu}
\address{Universit\'e Paris-Saclay, CNRS,  Laboratoire de math\'ematiques d'Orsay, 91405, Orsay, France}
\email{andrei.moroianu@math.cnrs.fr}

\author{Alberto Raffero}
\address{Dipartimento di Matematica ``G. Peano'' \\ Universit\`a degli Studi di Torino\\ Via Carlo Alberto 10\\10123 Torino\\ Italy}
\email{alberto.raffero@unito.it}


\begin{document}
\begin{abstract}
We consider left-invariant (purely) coclosed G$_2$-structures  on 7-dimensional 2-step nilpotent Lie groups. 
According to the dimension of the commutator subgroup, we obtain various criteria characterizing the Riemannian metrics induced by left-invariant purely coclosed G$_2$-structures. 
Then, we use them to determine the isomorphism classes of 2-step nilpotent Lie algebras admitting such type of structures. 
As an intermediate step, we show that every metric on a 2-step nilpotent Lie algebra admitting coclosed G$_2$-structures is induced by one of them. 
Finally, we use our results to give the explicit description of the metrics induced by purely coclosed G$_2$-structures on 2-step nilpotent Lie algebras with derived algebra of dimension at most two, 
up to automorphism. 
\end{abstract}
\maketitle

%%%%%%%%%%%%%%%%%%%%%%%%%%%%%%%%%%%%%%%%%%%%%%%%%%%%%%%%%%%%%%%%%%%%%%%%%%%%%%%%%%%%%%%%%
%%%%%%%%%%%%%%%%%%%%%%%%%%%%%%%%%%%%%%%%%%%%%%%%%%%%%%%%%%%%%%%%%%%%%%%%%%%%%%%%%%%%%%%%%
%																 INTRODUCTION
%%%%%%%%%%%%%%%%%%%%%%%%%%%%%%%%%%%%%%%%%%%%%%%%%%%%%%%%%%%%%%%%%%%%%%%%%%%%%%%%%%%%%%%%%
%%%%%%%%%%%%%%%%%%%%%%%%%%%%%%%%%%%%%%%%%%%%%%%%%%%%%%%%%%%%%%%%%%%%%%%%%%%%%%%%%%%%%%%%%
\section{Introduction}

A G$_2$-structure on a $7$-dimensional manifold $M$ is given by a $3$-form $\f\in\Omega^3(M)$ whose stabilizer at each point of $M$ is isomorphic to the automorphism group G$_2$ 
of the octonion algebra $\mathbb{O}$. By \cite{Gra}, $M$ admits G$_2$-structures if and only if its first and second Stiefel-Whitney classes vanish. 
Any G$_2$-structure $\f$ induces a metric $g_\f$ and an orientation on $M,$ and thus a Hodge duality operator $*_\f$. 

A G$_2$-structure $\f$ is said to be {\em purely coclosed} if it satisfies the conditions 
\begin{equation}\label{PCCIntro}
\d*_\f\f=0,\quad \d\f\W\f=0. 
\end{equation}
The equations \eqref{PCCIntro} characterize the pure class $\mathcal{W}_3$ in Fern\'andez-Gray's classification of G$_2$-structures \cite{FeGr} (see also \cite{Bry,ChSa}), 
while the condition $\d*_\f\f=0$ determines the wider class of {\em coclosed} G$_2$-structures $\mathcal{W}_1\oplus\mathcal{W}_3$. 
Remarkably, the latter are known to exist on every compact 7-manifold admitting G$_2$-structures  by an $h$-principle argument \cite{CrNo}. 
However, since this method is not constructive, different techniques are needed to obtain explicit examples. As for purely coclosed G$_2$-structures, no similar existence result is currently available. 

\smallskip 

The intrinsic torsion of a purely coclosed G$_2$-structure $\f$ can be identified with the 3-form $*_\f\d\f$, and so it vanishes identically if and only if $\f$ is closed (cf.~\cite{Bry}). 
When this happens, the Riemannian metric $g_\f$ induced by $\f$ is Ricci-flat and the corresponding Riemannian holonomy group is a subgroup of G$_2$. 

\smallskip

In theoretical physics, purely coclosed G$_2$-structures are closely related to the G$_2$-{\em Stromin\-ger system} of equations that arises considering the Killing spinor equations of 
10-dimensional string theory \cite{Str} on a compact 7-manifold (see e.g.~\cite{CGT,Iva} for more details and for the complete description of this system). 
Indeed, by \cite{FrIv,FrIv2}, the gravitino and dilatino Killing spinor equations with dilaton function $f$ on a compact 7-manifold $M$ are equivalent to the following 
system of equations for a G$_2$-structure $\f$ on $M$:
\[
\d*_\f\f = -2\,\d f\W*_\f\f,\quad \d\f\W\f=0.
\]
Any G$_2$-structure $\f$ satisfying them gives rise to a purely coclosed one via the global conformal change $ e^{\frac32f}\f$. Moreover, $\f$ itself is purely coclosed whenever $f$ is constant. 
Therefore, producing examples of compact 7-manifolds admitting purely coclosed G$_2$-structures constitutes an essential step towards the resolution of the G$_2$-Strominger system. 

\smallskip

Solutions to the G$_2$-Strominger system have been recently obtained in \cite{CGT} on $\mathbb{T}^3$-bundles over $K3$ surfaces. 
Previously, two examples of solutions with constant dilaton function were described in \cite{FIUV}. In these last two examples, the 7-manifold is the compact quotient of a 
simply connected nilpotent Lie group $N$ by a co-compact discrete subgroup (lattice) $\Gamma\subset N$, i.e., a {\em nilmanifold}, and the purely coclosed G$_2$-structure on $\Gamma\backslash N$ 
is induced by a left-invariant one on $N.$ Moreover, the Lie group $N$ is 2-step nilpotent, namely its Lie algebra $\gn$ is non-abelian and the corresponding derived algebra   
is contained in its center. In the non-compact setting, a further solution on a 2-step nilpotent Lie group was given in \cite{FIUVa}. 
It is worth observing that, when working with left-invariant G$_2$-structures on Lie groups,  the investigation can be done at the Lie algebra level, 
as left-invariant G$_2$-structures of a certain class on a Lie group are in one-to-one correspondence with G$_2$-structures of the same type on its Lie algebra.

\smallskip

In this paper, we carry out a systematic study of purely coclosed G$_2$-structures on 7-dimensional 2-step nilpotent Lie algebras, 
aimed at obtaining a classification of those admitting this type of G$_2$-structures, up to isomorphism, and characterizing the metrics induced by purely coclosed G$_2$-structures. 
 
\smallskip

We begin our investigation focusing on coclosed G$_2$-structures. Our first result is a refinement of the classification obtained in \cite{BFF}. 
There, the authors proved through a case-by-case study relying on the classification of 7-dimensional 2-step nilpotent Lie algebras, that each isomorphism class of such Lie algebras admits a coclosed 
$\G_2$-structure, with the exception of $\gn_{7,2,A}$ and $\gn_{7,2,B}$ (see the notation in Appendix \ref{2stepnilclass}, where we review the classification results obtained in \cite{Gon}). 
These last two Lie algebras are irreducible and have 2-dimensional derived algebra.
Here, we prove through direct arguments the following more precise statement which also takes into account the metric Lie algebra structure. 
\begin{theorem}\label{CCG2THM}
Let $\gn$ be a $7$-dimensional $2$-step nilpotent Lie algebra. If $\gn$ is irreducible and has $2$-dimensional derived algebra, then it carries no coclosed $\mathrm{G}_2$-structures. 
If $\gn$ is either reducible, or its derived algebra has dimension different from $2$, then every metric on $\gn$ is induced by a coclosed $\G_2$-structure. 
\end{theorem}

It follows from \cite{FeC,Lau,Wil} that every 7-dimensional 2-step nilpotent Lie algebra $\gn$ admits a (necessarily unique up to automorphism and scaling) {\em nilsoliton metric}, i.e., 
a metric $g$ whose Ricci endomorphism is of the form $\mathrm{Rc}(g) = \lambda\, \mathrm{Id} + D,$ for some $\lambda\in\R$ and some derivation $D$ of $\gn$. 
Nilsolitons correspond to left-invariant Ricci soliton metrics on nilpotent Lie groups \cite{Lau0} and so they constitute a generalization of Einstein metrics, that cannot exist on 
non-abelian nilpotent Lie groups by \cite[Thm.~2.4]{Mil}. 
Using the above observation together with Theorem \ref{CCG2THM}, we obtain a direct proof of \cite[Thm.~6.1, Thm.~6.3]{BFF}. 

\begin{corollary}\label{cornils}
Any $7$-dimensional $2$-step nilpotent Lie algebra admitting coclosed $\G_2$-struc\-tures has a coclosed $\G_2$-structure inducing the nilsoliton metric.
\end{corollary}

We then focus on purely coclosed G$_2$-structures on 2-step nilpotent Lie algebras.  
According to the dimension of the derived algebra, we obtain criteria for a given metric $g$ to be induced by such a structure. 
In order to state them here, we recall that for any 2-step nilpotent metric Lie algebra $(\gn,g)$ with derived algebra $\gn'$ and $g$-orthogonal decomposition $\gn=\gr\oplus\gn'$, 
the Chevalley-Eilenberg differential 
$\d:\mn^*\rightarrow \Lambda^2\mn^*$ vanishes on $\mr^*$ and defines an injection $j$ from $\mn'$ into $\so(\mr)\simeq \Lambda^2\mr^*$ (we refer the reader to Sect.~\ref{Sect2stepNil} for the precise details). 
In particular, $\dim (\mn')\leq \dim (\so(\mr))$, so when $\gn$ is 7-dimensional, this implies that the dimension of $\gn'$ is at most 3. 
\begin{theorem}\label{th} 
Let $(\mn,g)$ be a $7$-dimensional $2$-step nilpotent metric Lie algebra with derived algebra $\gn'$, and consider the $g$-orthogonal decomposition $\mn=\mr\oplus\mn'$.
\begin{enumerate}[$($i$)$]
\item\label{thi} If $\dim(\mn')=1$, there exists a purely coclosed $\G_2$-structure on $\gn$ inducing the metric $g$ if and only if $\tr^2(j(z)^2)=4\tr(j(z)^4)$ for every $z\in \mn'$.
\item\label{thii} If $\dim(\mn')=2$, there exists a purely coclosed $\G_2$-structure on $\gn$ inducing the metric $g$ if and only if 
there exists an oriented $4$-dimensional subspace $\tilde\gr\subset \gr$ with $\d(\mn')^*\subset \Lambda^2\tilde\mr^*$ such that for every orthonormal basis $\{\zeta_1,\zeta_2\}$ of $(\mn')^*$, 
the self-dual components $\d\zeta_1^+,\ \d\zeta_2^+\in \Lambda^2_{\sst+}\tilde\gr^*$ of $\d\zeta_1,\ \d\zeta_2\in \Lambda^2\tilde \gr^*$ are orthogonal and have equal norms.
\item\label{thiii} If $\dim(\mn')=3$, there exists a purely coclosed $\G_2$-structure on $\gn$ inducing the metric $g$ if and only if for some orientation of the $4$-dimensional space $\gr$, 
and for every orthonormal basis $\{\zeta_1,\zeta_2,\zeta_3\}$ of $(\mn')^*$, the Gram matrix of the self-dual components  of their differentials in 
$\Lambda^2_{\sst+}\gr^*$, $(S_{ij}) \coloneqq (g( \d\zeta_i^+,\d\zeta_j^+))$, satisfies $\tr^2(S)=2\tr(S^2)$. 
\end{enumerate}
\end{theorem}

As a consequence of this this result, we obtain the classification of all 7-dimensional 2-step nilpotent Lie algebras admitting purely coclosed G$_2$-structures, up to isomorphism. 
\begin{theorem} 
A $7$-dimensional $2$-step nilpotent Lie algebra $\mn$ admits purely coclosed $\G_2$-structures if and only if $\mn$ is not isomorphic to $\gh_3\oplus\bR^4$, $\gn_{7,2,A}$ or $\gn_{7,2,B}$, 
where $\gh_3$ denotes the $3$-dimensional Heisenberg Lie algebra.
\end{theorem}

Theorem \ref{th}, combined with metric classification results from \cite{DiSc,ReVi}, allows us to give the explicit description, up to automorphisms, 
of the metrics induced by purely coclosed G$_2$-structures on every 2-step nilpotent Lie algebra admitting such structures and with derived algebra of dimension at most two. 
In the remaining case, where $\dim(\gn')=3$, the lack of classification of metric Lie algebra structures prevents us from obtaining a similar result. 
Nevertheless, we show that each of these Lie algebras carries purely coclosed G$_2$-structures but also metrics which are not induced by any of them.

Finally, for each 2-step nilpotent Lie algebra of dimension 7, we are able to determine whether its nilsoliton metric is induced by a purely coclosed $\G_2$-structure  
(see Corollaries \ref{ricci1}, \ref{ricci2}, \ref{ricci3}).

It is worth stressing that all Lie groups corresponding to the Lie algebras carrying purely coclosed G$_2$-structures admit a lattice (cf.~\cite{Mal}). 
Therefore, the results above provide many new examples of (compact) manifolds where the G$_2$-Strominger system may be investigated.   

\smallskip

The paper is organized as follows. 
In Sect.~\ref{sectPCCG2}, we review some preliminaries on G$_2$-structures, and in Sect.~\ref{Sect2stepNil} we recall the main properties of 
2-step nilpotent metric Lie algebras.  Theorem \ref{CCG2THM} and Theorem \ref{th} are proved in Sect.~\ref{S5}. The discussion is divided into three parts according to the dimension of the derived algebra $\gn'$. 
Finally, in Sect.~\ref{metricn2} we describe the metrics induced by purely coclosed G$_2$-structures on 2-step nilpotent metric Lie algebras $\gn$ with $\dim(\gn')\leq 2$, 
and for each Lie algebra in the remaining case $\dim(\gn')=3$ we construct purely coclosed G$_2$-structures, as well as metrics which are not compatible with any purely coclosed G$_2$-structure.

%%%%%%%%%%%%%%%%%%%%%%%%%%%%%%%%%%%%%%%%%%%%%%%%%%%%%%%%%%%%%%%%%%%%%%%%%%%%%%%%%%%%%%%%%
%%%%%%%%%%%%%%%%%%%%%%%%%%%%%%%%%%%%%%%%%%%%%%%%%%%%%%%%%%%%%%%%%%%%%%%%%%%%%%%%%%%%%%%%%
%																 COCLOSED G2
%%%%%%%%%%%%%%%%%%%%%%%%%%%%%%%%%%%%%%%%%%%%%%%%%%%%%%%%%%%%%%%%%%%%%%%%%%%%%%%%%%%%%%%%%
%%%%%%%%%%%%%%%%%%%%%%%%%%%%%%%%%%%%%%%%%%%%%%%%%%%%%%%%%%%%%%%%%%%%%%%%%%%%%%%%%%%%%%%%%
\section{Preliminaries on G$_2$-structures}\label{sectPCCG2}
\subsection{Basic definitions}\label{BasicDef}  
A G$_2$-structure on a 7-dimensional vector space $V$ is defined by a 3-form $\f\in\Lambda^3V^*$ satisfying the non-degeneracy condition
\begin{equation}\label{NonDegG2}
v\co \f \W v \co \f \W \f \neq 0, \quad \forall~v\in V\smallsetminus\{0\}.
\end{equation}
Since the stabilizer $\GL(V)_\f\subset\GL(V)$ of any such 3-form is isomorphic to the exceptional Lie group G$_2$,
 the set $\Lambda^3_{\sst+}V^*$ of all G$_2$-structures on $V$ is isomorphic to $\GL(7,\R)/\G_2$ and thus open in $\Lambda^3V^*$.  

A G$_2$-structure $\f\in\Lambda^3_{\sst+}V^*$ gives rise to a unique inner product $g_\f$ and orientation on $V$ with corresponding volume form $\vol_\f$ satisfying 
\begin{equation}\label{metvolG2}
g_\f(v,w)\vol_\f = \frac16\, v\co\f\W w\co\f\W\f. 
\end{equation}
Moreover, there exists a $g_\f$-orthonormal basis $\mathcal{B} = \{e_1,\ldots,e_7\}$ of $V$ with dual basis $\mathcal{B}^*=\{e^1,\ldots,e^7\}$ such that 
\begin{equation}\label{G2adapted}
\begin{split}
\f 		&= e^{127}+e^{347}+e^{567} + e^{135}-e^{146}-e^{236}-e^{245}, \\
*_\f\f		&= e^{1234}+e^{1256}+e^{3456} +e^{1367}+e^{1457}+e^{2357}-e^{2467},
\end{split}
\end{equation}
where $*_\f$ is the Hodge operator determined by $g_\f$ and $\vol_\f$, and $e^{ijk\cdots}$ is a shorthand for the wedge product of covectors $e^i\W e^j\W e^k\W\cdots$. 
We shall call both $\mathcal{B}$ and $\mathcal{B}^*$ {\em adapted} bases to the G$_2$-structure $\f$. 

On the other hand, given an inner product $g$ on $V,$ we can consider a $g$-orthonormal basis $\mathcal{B}$ of $V$ and the G$_2$-structure $\f\in\Lambda^3_{\sst+}V^*$ 
having $\mathcal{B}$ as an adapted basis. By \eqref{metvolG2}, the metric $g_\f$ induced by $\f$ coincides with $g$. 
We shall refer to such $\f$ as the G$_2$-structure {\em induced} by the basis $\mathcal{B}$. 
Hence, there is a surjective map 
\[
\mathcal{G} : \Lambda^3_{\sst+}V^* \rightarrow \mathcal{S}^2_{\sst+}V^*,\quad \mathcal{G}(\f) = g_\f, 
\]
which is not injective, as the set of all G$_2$-structures inducing the same metric is parametrized by $\SO(7)/\G_2\cong\R P^7$ (see \cite[Remark 4]{Bry} for an explicit description). 
The G$_2$-structures belonging to $\mathcal{G}^{-1}(g)$ will be called {\em compatible} with $g$. 

\medskip

Consider now a G$_2$-structure $\f$ on $V,$ let $z\in V$ be a unit vector and denote by $W$ the 6-dimensional $g_\f$-orthogonal complement of $\langle z\rangle\subset V.$ 
Then, the G$_2$-structure $\f$ induces an SU(3)-structure $(h,J,\omega,\psip,\psim)$ on $W$ by means of the identities
\[
\f = \omega\W z^\flat + \psip,\quad *_\f\f = \frac12\,\omega\W\omega + \psim \W z^\flat,\quad g_\f = h + z^\flat \otimes z^\flat,
\]
where $z^\flat\in V^*$ denotes the $g_\f$-dual covector of $z$. 
Recall that the non-degenerate 2-form $\omega$ and the 3-forms $\psip,\psim$ satisfy the {\em compatibility condition} $\omega\W\psi_{\sst\pm}=0$ and the {\em normalization condition} 
\[
\psip\W\psim=\frac23\,\omega^3 = 4\vol_h,
\]
where $\vol_h$ is the volume form of the inner product $h$. 
Moreover, the $h$-orthogonal complex structure $J\in\End(W)$ is related to $h$ and $\omega$ via the identity $\omega = h(J\cdot,\cdot)$. 
Finally, there exists an adapted basis $\mathcal{B} = \{e_1,\ldots,e_7\}$ to $\f$ with $e_7=z$ and such that $\{e_1,\ldots,e_6\}$ is an $h$-orthonormal basis of $W$ which is {\em adapted} to 
the SU(3)-structure, that is to say
\begin{equation}\label{SU3adapted}
\omega = e^{12}+e^{34}+e^{56},\quad
\psip	= e^{135}-e^{146}-e^{236}-e^{245},\quad
\psim = e^{136}+e^{145}+e^{235}-e^{246}, 
\end{equation}
and $J(e_{2k-1}) = e_{2k}$, $k=1,2,3$. As before, any orthonormal basis $\{e_1,\ldots, e_6\}$ of $W$ {\em induces} an SU(3)-structure, 
namely, the structure  defined by \eqref{SU3adapted} in the given basis.

This procedure can be reversed, allowing one to obtain a G$_2$-structure on the 1-dimensional extension of a 6-dimensional vector space $W$ endowed with an SU(3)-structure $(h,J,\omega,\psi_\pm)$. 
In detail, if $\{e_1,\ldots,e_6\}$ is a basis of $W$ which is adapted to the SU(3)-structure, then the 7-dimensional vector space $V = W\oplus\langle z \rangle$ is endowed with a G$_2$-structure 
$\f$ having $\{e_1,\ldots,e_6,z\}$ as an adapted basis.

\subsection{(Purely) coclosed G$_{\mathbf2}$-structures}\label{G2mfds} 
Let $M$ be a 7-manifold endowed with a G$_2$-structure $\f\in\Omega^3(M)$.  
By \cite[Prop.~1]{Bry}, there exist unique differential forms $\tau_0\in\mathcal{C}^\infty(M)$, $\tau_1\in\Omega^1(M)$, 
$\tau_2\in\Omega^2_{14}(M)\coloneqq\{\alpha\in\Omega^2(M) \st \alpha \W *_\f\f =0\}$, and $\tau_3\in\Omega^3_{27}(M)\coloneqq \{\gamma\in\Omega^3(M) \st \gamma\W\f= 0,~\gamma\W*_\f\f=0\}$ such that 
\[
\begin{split}
\d\f &= \tau_0\,*_\f\f+3 \tau_1\wedge \f+*_\f\tau_3,\\
\d*_\f\f&=4 \tau_1\wedge *_\f\f+\tau_2\W\f.
\end{split}
\]
These differential forms are called the {\em torsion forms} of the G$_2$-structure $\f$, as they completely determine its intrinsic torsion (see also \cite{FeGr}).  

\smallskip

A G$_2$-structure $\f$ is said to be {\em coclosed} if it satisfies the equation
\[
\d*_\f\f=0.
\]
In terms of the torsion forms, the above condition is equivalent to the vanishing of $\tau_1$ and $\tau_2$.  
Coclosed G$_2$-structures constitute the class $\mathcal{W}_1\oplus\mathcal{W}_3$ in Fern\'andez-Gray's classification of G$_2$-structures \cite{FeGr}. 
The ``pure'' subclasses $\mathcal{W}_1$, $\mathcal{W}_3$ are characterized by the vanishing of $\tau_3$ and $\tau_0$, respectively. 
In the former case, the coclosed G$_2$-structure is called {\em nearly parallel} and the associated metric $g_\f$ is Einstein with positive scalar curvature $\Scal(g_\f)=\frac{21}{8}(\tau_0)^2$. 
In the latter, the G$_2$-structure is called {\em purely coclosed}. 
Notice that the vanishing of $\tau_0$ is equivalent to the condition 
\[
\d\f\W\f=0, 
\]
as $\tau_0 =\frac17 *_\f(\d\f\W\f)$. 

\smallskip

Simple examples of 7-manifolds admitting (purely) coclosed G$_2$-structures can be obtained as follows. 
Let $N$ be a 6-dimensional manifold endowed with an SU(3)-structure $(h,J,\omega,\psi_\pm)$. 
Then, the product manifold $M=N \times \R$ is endowed with a G$_2$-structure defined by the non-degenerate 3-form 
\[
\f = \omega \W \d t +\psip,
\]
where $\d t$ denotes the global 1-form on $\R$. 
The Riemannian metric induced by $\f$ is $g_\f = h + \d t^2$ and the Hodge dual of $\f$ is given by $*_\f\f=\frac12\,\omega^2+\psim\W \d t$.  
Now, we have
\[
\begin{split}
\d*_\f\f &= \d \omega\W\omega + \d\psim\W \d t,\\
\d\f\W\f &=  \d \omega\W \d t\W\psip + \d\psip\W\omega\W \d t = -2\,\d\omega\W\psip\W \d t. 
\end{split}
\]
Thus, we immediately see that $\f$ is coclosed if and only if the SU(3)-structure satisfies the conditions 
\begin{equation}\label{HFSU3}
\d \omega\W\omega=0,\quad \d\psim=0.
\end{equation}
Moreover, $\f$ is purely coclosed if and only if the SU(3)-structure satisfies the additional condition 
\begin{equation}\label{special}
\d\omega\W\psip=0. 
\end{equation}
An SU(3)-structure satisfying the equations \eqref{HFSU3} is called {\em half-flat} (cf.~\cite{ChSa}).

%%%%%%%%%%%%%%%%%%%%%%%%%%%%%%%%%%%%%%%%%%%%%%%%%%%%%%%%%%%%%%%%%%%%%%%%%%%%%%%%%%%%%%%%%
%%%%%%%%%%%%%%%%%%%%%%%%%%%%%%%%%%%%%%%%%%%%%%%%%%%%%%%%%%%%%%%%%%%%%%%%%%%%%%%%%%%%%%%%%
%													PURELY COCLOSED G2 ON 2-STEP NILPOTENT
%%%%%%%%%%%%%%%%%%%%%%%%%%%%%%%%%%%%%%%%%%%%%%%%%%%%%%%%%%%%%%%%%%%%%%%%%%%%%%%%%%%%%%%%%
%%%%%%%%%%%%%%%%%%%%%%%%%%%%%%%%%%%%%%%%%%%%%%%%%%%%%%%%%%%%%%%%%%%%%%%%%%%%%%%%%%%%%%%%%
\section{The structure of 2-step nilpotent metric Lie algebras}\label{Sect2stepNil}

We now consider the case when the 7-dimensional manifold is a Lie group $N$ endowed with a left-invariant G$_2$-structure, namely a non-degenerate 3-form $\f\in\Omega^3(N)$ 
that is invariant by left translations of $N$. In this case, the Riemannian metric $g_\f$ induced by $\f$ is also left-invariant.
 
The identification of the Lie algebra $\gn$ of $N$ with the tangent space to $N$ at the identity gives rise to a one-to-one correspondence between left-invariant tensors on $N$ 
and algebraic tensors of the same type defined on $\gn$. In particular, left-invariant Riemannian metrics on $N$ correspond to inner products on $\mn$, 
and left-invariant  G$_2$-structures on $N$ correspond to G$_2$-structures $\f$ on $\gn$, i.e., $\f\in\Lambda^3_{\sst+}\gn^*$. The conditions for $\f$ to be purely coclosed read $\d*_\f\f=0$ and $\d\f\W\f=0$,
where $\d$ denotes the Chevalley-Eilenberg differential of $\gn$.

\smallskip
Along this paper, a {\em metric Lie algebra} is the data of a real Lie algebra $\gn$ endowed with an inner product $g$. 
We focus on the case when $\gn$ is 2-step nilpotent, namely when $\gn$ is not abelian and $\ad_x^2=0$ for all $x\in\gn$, where $\ad$ denotes the adjoint map of $\mn$. 
Under this assumption, the structure of a metric Lie algebra $(\gn,g)$ can be described as follows (see \cite{Ebe} for more details). 

Denote by $\gn'\coloneqq[\gn,\gn]$ the derived algebra of $\mn$ and by $\gz$ its center.
As $\gn$ is 2-step nilpotent, we have $\{0\} \neq \gn'\subset \gz$. 
Let $\mr$ denote the $g$-orthogonal complement of $\gn'$ in $\gn$, 
so that $\gn = \gr \oplus \mn'$ as a direct sum of vector spaces. The  metric Lie algebra structure of $(\gn,g)$ is encoded into the injective linear map $j:\gn'\rightarrow\so(\gr)$ defined via the identity 
\begin{equation}\label{jz}
g(j(z)x,y) \coloneqq g(z,[x,y]),
\end{equation}
for all $z\in\gn'$ and $x,y\in\gr$.

Using the metric $g$, we can always identify $\gn$ with its dual Lie algebra $\gn^*$, and we can see the elements in $\Lambda^2\mn^*$ as skew-symmetric endomorphisms in $\so(\mn)$. 
Under these identifications, it is straightforward to check that $\d x^\flat =0$ if $x\in \mr$, where $x^\flat$ denotes the metric dual of $x$. 
In addition, for any $z\in\gn'$ the 2-form $\d z^\flat$ belongs to the subspace $\Lambda^2\mr^*$ of $\Lambda^2\mn^*$ and it corresponds to the skew-symmetric endomorphism $-j(z)\in \so(\mr)$.

The previous discussion allows us to associate to any 2-step nilpotent metric Lie algebra $(\mn,g)$, the following data: the inner  product spaces $(\mr,g_\mr)$ and $(\mn',g_{\mn'})$ together with an injection 
$j:\mn'\rightarrow \so(\mr)$.  Here and henceforth, we denote by $g_\mk$ the restriction of the metric $g$ to the subspace $\mk$ of $\mn$. 

Conversely, given two inner product spaces $(\mr,g_\mr)$ and $(\mn',g_{\mn'})$, together with an injective linear map $j:\mn' \rightarrow \so(\mr)$, 
we can define a 2-step nilpotent metric Lie algebra $(\mn,g)$ as follows. 
We set $\mn\coloneqq\mr\oplus \mn'$, we endow it with the metric $g=g_\mr+g_{\mn'}$ and we define the Lie bracket on $\mn$ so that the  elements in $\mn'$ are in the center, 
it satisfies $[\mr,\mr]\subset \mn'$ and it is determined by
\begin{equation*}
g_{\mn'}(z,[x,y]) \coloneqq g_\mr(j(z)(x),y), \quad  \mbox{for all } x,y\in \mr,\; z\in \mn'.
\end{equation*}
It is easy to verify that $\mn$ is a $2$-step nilpotent Lie algebra with derived algebra $\mn'$ (which justifies the initial notation).
\smallskip

For any 2-step nilpotent metric Lie algebra $(\mn,g)$ we have $\mn'\subset \mz$. 
Let $\ma$ denote the orthogonal complement of $\mn'$ inside $\mz$. It is straightforward to check from \eqref{jz} that $\ma$ is the common kernel of the endomorphisms $j(z) \in \so(\mr)$, 
when $z$ runs through $\mn'$. 
In addition, for  any $x\in \ma$, the orthogonal complement $\tilde\mn\coloneqq\lela x\rira^\bot$ is an ideal of $\mn$, which now decomposes as a direct sum of orthogonal ideals
$(\mn,g) = (\tilde\mn,g_{\tilde\mn}) \oplus (\lela x\rira,g_{\lela x\rira})$. 
\smallskip

Nilpotent Lie algebras of dimension 7 are classified up to isomorphism (see \cite{Gon}); we recall the classification of those which are real and 2-step nilpotent in Appendix \ref{2stepnilclass}. 
Throughout the paper, the structure equations of an $n$-dimensional Lie algebra $\gn$ are written with respect to a basis of covectors $\{f^1,\ldots,f^n\}$ 
by specifying the $n$-tuple $(\d f^1,\ldots, \d f^n)$.

%%%%%%%%%%%%%%%%%%%%%%%%%%%%%%%%%%%%%%%%%%%%%%%%%%%%%%%%%%%%%%%%%%%%%%%%%%%%%%%%%%%%%%%%%
%%%%%%%%%%%%%%%%%%%%%%%%%%%%%%%%%%%%%%%%%%%%%%%%%%%%%%%%%%%%%%%%%%%%%%%%%%%%%%%%%%%%%%%%%
%														MAIN THEOREMS PROOFS
%%%%%%%%%%%%%%%%%%%%%%%%%%%%%%%%%%%%%%%%%%%%%%%%%%%%%%%%%%%%%%%%%%%%%%%%%%%%%%%%%%%%%%%%%
%%%%%%%%%%%%%%%%%%%%%%%%%%%%%%%%%%%%%%%%%%%%%%%%%%%%%%%%%%%%%%%%%%%%%%%%%%%%%%%%%%%%%%%%%
\section{When is a metric induced by a (purely) coclosed G$_2$-structure?}\label{S5}

In this section, we will prove Theorem \ref{CCG2THM} and Theorem \ref{th}.
As we already observed, given a 7-dimensional 2-step nilpotent Lie algebra $\gn$, the possible dimensions of its derived algebra $\gn'$ are 1, 2 or 3. We shall discuss each case separately.

%%%%%%%%%%%%%%%%%%%%%%%%%%%%%%%%%%%%%%%%%%%%%%%%%%%%%%%%%%%%%%%%%%%%%%%%%%%%%%%%%%%%%%%%%
%%%%%%%%%%%%%%%%%%%%%%%%%%%%%%%%%%%%%%%%%%%%%%%%%%%%%%%%%%%%%%%%%%%%%%%%%%%%%%%%%%%%%%%%%
%																DIM n' = 1
%%%%%%%%%%%%%%%%%%%%%%%%%%%%%%%%%%%%%%%%%%%%%%%%%%%%%%%%%%%%%%%%%%%%%%%%%%%%%%%%%%%%%%%%%
%%%%%%%%%%%%%%%%%%%%%%%%%%%%%%%%%%%%%%%%%%%%%%%%%%%%%%%%%%%%%%%%%%%%%%%%%%%%%%%%%%%%%%%%%
\subsection{Case 1: $\dim(\gn')=1$}\label{subsec:case1}

Let $(\mn,g)$ be a 2-step nilpotent metric Lie algebra of dimension 7 with 1-dimensional derived algebra $\mn'$ and consider the  $g$-orthogonal splitting $\gn = \gr\oplus\gn'$, where $\gr=(\gn')^\perp$.  
Given a unit vector $z\in\gn'$, the structure equations of $\gn$ are completely determined by the differential of the metric dual $z^\flat$ of $z$, which we denote by $\alpha\coloneqq\d z^\flat\in\Lambda^2\gr^*$.  
Let $A \coloneqq -j(z)\in \so(\mr)$ denote the skew-symmetric endomorphism corresponding to $\alpha\in \Lambda^2\mr^*$.  

Notice that, depending on the rank of $A$, $\gn$ is isomorphic to one of $\gh_3\oplus\bR^4$ ($\rank(A)=2$), $\gh_5\oplus\bR^2$ ($\rank(A)=4$), $\gh_7$ ($\rank(A)=6$), 
where $\mh_{i}$ denotes the Heisenberg Lie algebra of dimension $i$.

Let  $\f$ be a G$_2$-structure on $\gn$ such that $g_\f=g$. Then the 6-dimensional vector subspace $\gr\subset\gn$ is endowed with an SU(3)-structure $(h,J,\omega,\psi_\pm)$ (see Sect.~\ref{BasicDef}). 
In particular, we can write 
\begin{equation}
\label{eq:ph1}
\f = \omega\W z^\flat + \psip, \quad *_\f\f = \frac12\,\omega^2+\psim\W z^\flat. 
\end{equation}
Using the complex structure $J$ of $\gr$, we can give the following characterization.  

\begin{proposition}\label{pro:JA}
The $\G_2$-structure $\f$ in \eqref{eq:ph1} is coclosed if and only if $JA=AJ$ and, moreover, it is purely coclosed if and only if $\tr(JA)=0$.
\end{proposition}
\begin{proof}
Since $\omega$ and $\psim$ are forms on $\gr$, their differentials vanish,  so the G$_2$-structure $\f$ in \eqref{eq:ph1} is coclosed if and only if 
\[
0 = \d*_\f\f = -\psim\W\alpha. 
\]
Thus, $\alpha$ belongs to the kernel of the map $\cdot\W\psim:\Lambda^2\gr^*\rightarrow\Lambda^5\gr^*$, which is the space of real 2-forms of type $(1,1)$ with respect to $J$. 
In terms of the skew-symmetric endomorphism $A\in\so(\gr)$ corresponding to $\alpha$, this is equivalent to $JA = AJ.$ 
Moreover, the intrinsic torsion form $\tau_0$ of $\f$ vanishes if and only if  
\[
0 = \d\f \W \f = \omega^2 \W \alpha \W z^\flat,
\]
which is equivalent to the additional constraint $\tr(JA)=0$.
\end{proof}

Using Proposition \ref{pro:JA}, we can show the following result, which proves Theorem \ref{CCG2THM} for the case $\dim(\gn')=1$, and part (\ref{thi}) of Theorem \ref{th}.

\begin{proposition}\label{dimnn1}
Let $(\gn,g)$ be a $7$-dimensional $2$-step nilpotent metric Lie algebra with $\dim(\gn')=1$.  
\begin{enumerate}[$1)$]
\item There exists a coclosed $\G_2$-structure on $\gn$ inducing the metric $g$.
\item Furthermore, there exists a purely coclosed $\G_2$-structure on $\gn$ inducing the metric $g$ if and only if $\tr^2(j(z)^2)=4\tr(j(z)^4)$ for every $z\in \mn'$.
\end{enumerate}
\end{proposition}

\begin{proof}\
\begin{enumerate}[1)]
\item Let $g$ be an inner product on $\mn$ and consider a unit vector  $z$ in $\mn'$. Then, $A \coloneqq -j(z)\in\so(\gr)$ and there exists a $g$-orthonormal basis $\{e_1,\ldots,e_6\}$ of $\gr$ 
such that the matrix of $A$ 
with respect to this basis has the form
\begin{equation}\label{Abd}
A = 
\begin{psmallmatrix}
0&-a&0&0&0&0\\ \noalign{\medskip}
a&0&0&0&0&0\\ \noalign{\medskip}
0&0&0&-b&0&0\\ \noalign{\medskip}
0&0&b&0&0&0\\ \noalign{\medskip}
0&0&0&0&0&-c\\ \noalign{\medskip}
0&0&0&0&c&0
\end{psmallmatrix},
\end{equation}
where $a,b,c\in\mathbb{R}$ are not all zero. Consider the SU(3)-structure  $(h,J,\omega,\psi_\pm)$ on $\gr$ induced by this basis $\{e_1,\ldots,e_6\}$. 
Then, $\omega$ satisfies \eqref{SU3adapted} and the complex structure $J$ satisfies $JA=AJ$. 
Thus, the 3-form $\f = \omega\W z^\flat+\psip$ verifies \eqref{eq:ph1} and, by Proposition \ref{pro:JA}, 
it defines a coclosed G$_2$-structure on $\gn$ with $g_\f=g$.   
\item  Let $\f$ be a purely coclosed G$_2$-structure  on $\gn$ such that $g_\f=g$, and let $z$ be a unit vector  in $\mn'$. Then $\f$ induces an SU(3)-structure $(h,J,\omega,\psi_\pm)$ on $\gr$ satisfying  
$\f = \omega\W z^\flat+\psip$ and $\omega=g_{\gr}(J\cdot,\cdot)$ (see Sect.~\ref{BasicDef}). 
So $\f$ is of the form \eqref{eq:ph1} and, by Proposition \ref{pro:JA}, we obtain that $JA=AJ$ and $\tr(JA)=0$, for $A=-j(z)$.
The eigenspaces of the symmetric endomorphism $AJ$ are preserved by $J$ so they have even dimension. Consequently the spectrum of $AJ$ is of the form $(a,a,b,b,c,c)$ with $a+b+c=0$, whence 
\[
\tr(A^2)=-\tr((AJ)^2)=-2(a^2+b^2+(a+b)^2)=-4(a^2+b^2+ab),
\]
and thus 
\[
\tr(A^4)=\tr((AJ)^4)=2(a^4+b^4+(a+b)^4)=4(a^2+b^2+ab)^2=\frac14\tr^2(A^2).
\]
It is clear that for any other vector $z'=\lambda z$ in $\mn'$, the corresponding matrix $A'=-j(z')$ verifies $A' =\lambda A$, so  the same equation holds by homogeneity.

Conversely, suppose that $g$ is an inner product on $\mn$ for  which $\tr^2(j(z)^2)=4\tr(j(z)^4)$ holds for every $z$ in $\mn'$.
Consider again a unit vector $z\in\mn'$ and a $g$-orthonormal basis $\{e_1,\ldots,e_6\}$ of $\gr$ in which the matrix of $A=-j(z)$ is given by \eqref{Abd}.
Since $\tr^2(A^2)=4\tr(A^4)$, we can write
\begin{eqnarray*}0&=&\tr(A^4)-\frac14\tr^2(A^2)=2(a^4+b^4+c^4)-(a^2+b^2+c^2)^2\\
&=&a^4+b^4+c^4-2(a^2b^2+b^2c^2+c^2a^2)\\
&=&(a+b+c)(a+b-c)(a-b+c)(a-b-c).
\end{eqnarray*}
By permuting $e_1$ with $e_2$ and/or $e_3$ with $e_4$ if necessary, we can change the signs of $a$ and/or $b$, so from the above equation we can assume that $a+b+c=0$.
Let $(h,J,\omega,\psi_\pm)$ be the SU(3)-structure on $\gr$ induced by the basis $\{e_1,\ldots,e_6\}$. 
Then, we have $JA=AJ=\diag(-a,-a,-b,-b,-c,-c)$, so $\tr(JA)=-2(a+b+c)=0$, and therefore the 3-form $\f = \omega\W z^\flat+\psip$ defines a purely coclosed G$_2$-structure  on $\gn$ with $g_\f=g$.  
\end{enumerate}
\end{proof}

One can prove that the condition $\tr^2(j(z)^2)=4\tr(j(z)^4)$ cannot hold whenever $j(z)$ has rank 2, so we can state the following.

\begin{corollary}\label{cor:pcch3}
The Lie algebra $\gh_3\oplus\R^4$ does not admit any purely coclosed $\G_2$-structure. 
\end{corollary}
\begin{proof}
For every metric on $\gh_3\oplus\R^4$, the endomorphism $A=-j(z)$ has rank 2, for any non-zero $z\in \mn'$, so its square is proportional to an orthogonal projector on a 2-plane: 
$A^2=\lambda P$ with $\lambda\in\R\smallsetminus\{0\}$, $P^2=P$ and $\tr(P)=2$. Then $\tr(A^2)=2\lambda$, $\tr(A^4)=2\lambda^2$, so the equation $\tr^2(A^2)=4\tr(A^4)$ cannot hold.
\end{proof}

%%%%%%%%%%%%%%%%%%%%%%%%%%%%%%%%%%%%%%%%%%%%%%%%%%%%%%%%%%%%%%%%%%%%%%%%%%%%%%%%%%%%%%%%%
%%%%%%%%%%%%%%%%%%%%%%%%%%%%%%%%%%%%%%%%%%%%%%%%%%%%%%%%%%%%%%%%%%%%%%%%%%%%%%%%%%%%%%%%%
%																DIM n' = 2
%%%%%%%%%%%%%%%%%%%%%%%%%%%%%%%%%%%%%%%%%%%%%%%%%%%%%%%%%%%%%%%%%%%%%%%%%%%%%%%%%%%%%%%%%
%%%%%%%%%%%%%%%%%%%%%%%%%%%%%%%%%%%%%%%%%%%%%%%%%%%%%%%%%%%%%%%%%%%%%%%%%%%%%%%%%%%%%%%%%

\subsection{Case 2: $\dim(\gn')=2$}

In this section we consider coclosed G$_2$-structures on 7-dimen\-sional 2-step nilpotent Lie algebras $\gn$ with $\dim(\gn')=2$.
From the classification reviewed in Appendix \ref{2stepnilclass}, we know that any such Lie algebra is isomorphic either to one of the decomposable Lie algebras 
\[
\gn_{5,2}\oplus\bR^2,\quad \gh_3\oplus\gh_3\oplus\bR,\quad \gh_3^{\bC}\oplus\bR,\quad \gn_{6,2}\oplus\bR,
\]
or to one of the indecomposable Lie algebras $\gn_{7,2,A},\gn_{7,2,B}$. 

We will first show that the existence of coclosed G$_2$-structures on $\gn$ forces its decomposability. More precisely, we have:

\begin{proposition} \label{decomp}
Let $\gn$ be a $7$-dimensional $2$-step nilpotent  Lie algebra with $2$-dimensional derived algebra $\gn'$ and let $\f$ be a coclosed $\mathrm{G}_2$-structure on $\gn$.  
If $z_1,z_2$ is any $g_\f$-orthonormal basis of $\gn'$, then the (unit length) vector $x\coloneqq\f(\cdot, z_1,z_2)^\sharp$ belongs to $\ga$. 
\end{proposition}
\begin{proof}
Let $\f$ be a G$_2$-structure on $\gn$ and let $\{z_1,z_2\}$ be  an  orthonormal basis of $\mn'$. 
As G$_2$ acts transitively on ordered pairs of orthonormal vectors on $\bR^7$, there exists a basis $\{e_1,\ldots,e_7\}$ of $\gn$ adapted to $\f$ such that 
$z_1=e_5$ and $z_2=e_6$. Then, $\f$ and $*_\f\f$ can be written as in \eqref{G2adapted}, so $\f(\cdot, e_5,e_6)=e^7$. 
Consequently, the vector $x$ defined above is equal to $e_7$ 
and the thesis is equivalent to showing that $e_7\lrcorner\, \d e^5=e_7\lrcorner\,\d e^6=0$.

Consider the subspace $\tilde \gr\coloneqq \langle e_1,\ldots,e_4\rangle\subset \gr=\langle e_1,\ldots,e_4,e_7\rangle$. 
Then there exist $\alpha_k\in\Lambda^2\tilde \gr^*$ and $\theta_k\in\tilde\gr^*$, $k=5,6$, such that the structure equations of $\gn$ with respect to the basis $\{e^1,\ldots,e^7\}$ of $\gn^*$ are the following
\begin{equation}\label{str2}
\begin{dcases}
\d e^k=0,~k=1,2,3,4,7\\
\d e^k = \alpha_k + \theta_k \W e^7,~k=5, 6.
\end{dcases}
\end{equation}
We thus have $e_7\lrcorner\, \d e^k=-\theta_k,$  so our aim is to show that $\theta_k=0$ for $k=5,\ 6$.

We orient the subspace $\tilde \gr\subset\gn$ by the volume form $e^{1234}$. 
The space $\Lambda^2\tilde\gr^*$ can be decomposed into the orthogonal direct sum of the subspaces 
of self-dual forms $\Lambda^2_{\sst+}\tilde\gr^*\coloneqq \{\sigma\in\Lambda^2\tilde\gr^* \st *\sigma=\sigma \}$ 
and anti-self-dual forms $\Lambda^2_{\sst-}\tilde\gr^*\coloneqq \{\sigma\in\Lambda^2\tilde\gr^* \st *\sigma=-\sigma \}$, where $*$ denotes the Hodge operator on $\tilde\gr$. 
We choose the following $g_\f$-orthogonal basis of $\Lambda^2_{\sst+}\tilde\gr^*$
\begin{equation}\label{sigma}
\sigma_1 = e^{13}-e^{24},\quad \sigma_2 =- e^{14}-e^{23},\quad \sigma_3 = e^{12}+e^{34}. 
\end{equation}
Note that $|\sigma_k|=\sqrt{2}$, for $k=1,2,3$. By \eqref{G2adapted}, we have
\begin{equation}\label{g2pro}
\begin{split}
\f&=\sigma_1\W e^5+\sigma_2\W e^6+\sigma_3\W e^7+e^{567},\\
*_\f\f	&= e^{1234} + \sigma_1 \W e^{67} +\sigma_2 \W e^{75}+\sigma_3 \W e^{56}.
\end{split}
\end{equation}
Now, using \eqref{str2}, we obtain
\[
\begin{split}
\d*_\f\f	&=  \sigma_1 \W \d (e^{67}) +\sigma_2 \W \d(e^{75}) +\sigma_3 \W \d(e^{56}) \\
		&= (\sigma_1\W\alpha_6- \sigma_2 \W \alpha_5) \W e^7 +\sigma_3\W\alpha_5 \W e^6 -\sigma_3 \W \theta_5 \W e^{67} \\
		&\quad -\sigma_3 \W \alpha_6 \W e^5 +\sigma_3 \W \theta_6 \W e^{57}. 
\end{split}
\]
Therefore, $\f$ is coclosed if and only if the following equations on $\tilde\gr$ hold: 
\begin{equation}\label{n21}
\sigma_1\W\alpha_6- \sigma_2 \W \alpha_5 = 0 = \sigma_3\W\alpha_5 = \sigma_3 \W \alpha_6,
\end{equation}
\begin{equation} \label{n22}
\sigma_3\W\theta_5 = 0 = \sigma_3\W\theta_6.
\end{equation}

As $\sigma_3$ is a non-degenerate 2-form on $\tilde\gr$, the equations \eqref{n22} imply that $\theta_5=\theta_6=0$ as claimed.
\end{proof}

\begin{remark}
Recall that $\gn_{7,2,A}$ and $\gn_{7,2,B}$ are the only indecomposable $2$-step nilpotent Lie algebras of dimension 7 having $\dim(\mn')=2$.
Proposition \ref{decomp} thus provides an alternative proof of the fact that $\gn_{7,2,A}$ and $\gn_{7,2,B}$ do not admit coclosed $\G_2$-structures.  
This result was already proved in \cite[Thm.~5.1]{BFF} by a case by case analysis using the classification of 7-dimensional $2$-step nilpotent Lie algebras.
\end{remark}

\begin{corollary} \label{pro:n2cocl}  
Let $\mn$ be a $7$-dimensional $2$-step nilpotent Lie algebra with $\dim (\mn')=2$, 
let $\{e^1, \ldots, e^7\}$ be a basis of $\gn^*$ for which the structure equations of $\gn$ are given by \eqref{str2} and consider the metric $g$ making this frame orthonormal. Endow the $4$-dimensional subspace 
$\tilde{\gr}=\langle e_1,e_2,e_3,e_4\rangle\subset\gn$ with the metric $g_{\tilde{\gr}}$ and the orientation $e^{1234}$. 
Then, the $\G_2$-structure $\f$ induced by $\{e^1, \ldots, e^7\}$ is coclosed if and only if $\theta_5=\theta_6=0$ and there exist $a,b,c\in\R$ such that the self-dual parts $\alpha_5^+$, $\alpha_6^+$ 
of the forms $\alpha_5, \alpha_6\in\Lambda^2\tilde{\gr}^*$  satisfy
\begin{equation}\label{n2coclosed}
\begin{dcases}
\alpha_5^+ = a\,\sigma_1 + b\,\sigma_2, \\
\alpha_6^+ = b\,\sigma_1 +c\, \sigma_2,
\end{dcases}
\end{equation}
where $\sigma_1,\sigma_2$ are defined in \eqref{sigma}.
In addition, $\varphi$ is purely coclosed if and only if $a+c=0$.
\end{corollary}

\begin{proof} 
Let $\{e^1,\ldots,e^7\}$ be a basis of $\gn^*$  for which the structure equations are given by \eqref{str2}, and let $g$ denote the metric making this basis orthonormal.  
Then, the $\G_2$-structure $\f$ induced by $\{e^1, \ldots, e^7\}$ satisfies \eqref{g2pro}, where $\{\sigma_1,\sigma_2,\sigma_3\}$ is the basis of self-dual forms on $\tilde\mr$ defined in \eqref{sigma}.

From the proof of Proposition \ref{decomp}, we have that $\f$ is coclosed if and only if \eqref{n21} and \eqref{n22} hold. The latter is equivalent to the vanishing of $\theta_5$ and $\theta_6$.
As $\{\sigma_1,\sigma_2,\sigma_3\}$ is an orthogonal basis of  $\Lambda^2_{\sst+}\tilde\gr^*$, \eqref{n21} imposes constraints on the self-dual part $\alpha_k^+$ of $\alpha_k$, $k=5,6$. 
More precisely, the equations in \eqref{n21} hold if and only if the components of $\alpha^+_5$ and $\alpha^+_6$ along $\sigma_3$ vanish and $g_\f(\alpha^+_5,\sigma_2) =g_\f(\alpha^+_6,\sigma_1)$, that is, 
if and only if there exist some real numbers $a,b,c,$ such that  \eqref{n2coclosed} is verified.

When $\f$ is purely coclosed, there is an additional constraint in the system \eqref{n2coclosed}. In detail, we have 
\[
\begin{split}
\d\f\W\f	&= \left(\sigma_1\W\alpha_5+\sigma_2\W\alpha_6+
\alpha_5 \W e^{67}-\alpha_6\W e^{57}\right) \W \f \\
		&= 2\left(\sigma_1\W\alpha_5+\sigma_2\W\alpha_6\right)\W e^{567}. 
\end{split}
\]
Thus, $\d\f\W\f=0$ if and only if $g_\f(\sigma_1,\alpha_5) +g_\f(\sigma_2,\alpha_6)=0$, namely if and only if $a+c=0$ in \eqref{n2coclosed}.
\end{proof}

Using Corollary \ref{pro:n2cocl}, we can show the following result which constitutes Theorem \ref{CCG2THM} for the case $\dim(\gn')=2$, and part (\ref{thii}) of Theorem \ref{th}.

\begin{proposition}\label{dim2} 
Let $(\gn,g)$ be a $7$-dimensional $2$-step nilpotent metric Lie algebra  with $\dim(\gn')=2$. 
\begin{enumerate}[$1)$]
\item\label{dim2i} There exists a coclosed $\G_2$-structure on $\gn$ inducing the metric $g$ if and only if $\ga\ne 0$. 
\item\label{dim2ii} Furthermore, there exists a purely coclosed $\G_2$-structure on $\gn$ inducing the metric $g$ if and only if 
there exists an oriented $4$-dimensional subspace $\tilde\gr\subset \gr$ with $\d(\mn')^*\subset \Lambda^2\tilde\mr^*$ such that for every orthonormal basis $\{\zeta_1,\zeta_2\}$ of $(\mn')^*$, 
the self-dual components $\d\zeta_1^+,\ \d\zeta_2^+\in \Lambda^2_{\sst+}\tilde\gr^*$ of $\d\zeta_1,\ \d\zeta_2\in \Lambda^2\tilde \gr^*$ are orthogonal and have equal norms.
\end{enumerate}
\end{proposition}

\begin{proof}
We first prove the direct implication in both cases. Assume that $(\gn,g)$, with $\dim(\gn')=2$, has a coclosed $\G_2$-structure $\f$ such that $g_\f=g$. 
By the transitivity of G$_2$ on orthonormal pairs of vectors in $\gn$, there exists a basis $\{e_1,\ldots,e_7\}$ of $\gn$ adapted to $\f$ 
with $\gn' = \langle e_5,e_6 \rangle$. By Proposition \ref{decomp}, $x=e_7$ belongs to $\ga\subset\mr$, so $\ga\ne0$.
Moreover, the orthogonal complement $\tilde\gr$ of $x$ in $\gr$ is $\tilde\gr\coloneqq\langle e_1,\ldots,e_4\rangle$. 

Let us consider the orientation of $\tilde\gr$ given by $e^{1234}$.
If $\f$ is purely coclosed, by Corollary \ref{pro:n2cocl} there are some real numbers $a,b$ such that the self-dual parts $\alpha_k^+\in\Lambda^2_{\sst+}\tilde\gr^*$ of $\alpha_k\coloneqq\d e^k$, for $k=5,6,$ 
satisfy \eqref{n2coclosed} with $c=-a$ (recall that $\sigma_1,~\sigma_2,~\sigma_3$ are given by \eqref{sigma}). 
From this system it clearly follows that $\alpha_5^+$ and $\alpha_6^+$ are orthogonal and have the same length. 

Every other orthonormal coframe $\{\zeta_1, \zeta_2\}$ of $(\mn')^*$ differs from $\{e_5,\e_6\}$ by the action of 
an orthogonal matrix in ${\rm O}(2)$, which also describes the transformation taking the pair 
$\{\d \alpha_5^+, \d\alpha_6^+\}$ to $\{\d \zeta_1^+, \d\zeta_2^+\}$.
Therefore, $\{\d \zeta_1^+, \d\zeta_2^+\}$ are orthogonal and have equal norms.

\smallskip
We now prove the converse statements. 
\smallskip

\noindent
\ref{dim2i}) 
Assume that $\ga\ne 0$, let $ e_7\in\ga\subset \mr$ be a unit vector, and choose a $g$-orthonormal basis $\{e_5,e_6\}$ of $\gn'$. 
Consider the metric induced by $g$, fix some orientation on the orthogonal complement $\tilde\gr$ of $e_7$ in $\gr$, 
and let $\alpha_k\coloneqq\d e^k\in\Lambda^2\tilde\gr^*$, $k=5,6$. 

Let $\mathcal{P}\subset \Lambda^2_{\sst+}\tilde\gr^*$ be a plane containing the self-dual components $\alpha_5^+,\alpha_6^+$ of $\alpha_5$ and $\alpha_6$. 
Using the polar decomposition, one can find an orthonormal basis in $\mathcal{P}$ 
with respect to which the matrix of $\alpha_5^+,\alpha_6^+$ is symmetric. 
Consequently, one can find two orthogonal elements $\{\s_1,\s_2\}$ in $\Lambda^2_{\sst+}\tilde\gr^*$ with $|\s_1|=|\s_2|=\sqrt{2}$ such that the self-dual forms $\alpha_5^+,\alpha_6^+$ can be written as
\[
\begin{dcases}
\alpha_5^+ = a\,\sigma_1 + b\,\sigma_2, \\
\alpha_6^+ = b\,\sigma_1 +c \, \sigma_2,
\end{dcases}
\]
for some real numbers $a,b,c.$ 
Since $\SO(4)$ acts transitively on pairs of orthogonal forms of length $\sqrt2$ in $\Lambda^2_{\sst+}\tilde \gr^*$, we can always find an oriented orthonormal basis $\{e_1,\ldots,e_4\}$ of $\tilde\gr$ such that 
${\s}_1 = e^{13}-e^{24}$ and ${\s}_2= -e^{14}-e^{23}$. 
Then the G$_2$-structure induced by the orthonormal basis $\{e_1, \ldots,e_7\}$ of $(\mn,g)$ is coclosed by Corollary \ref{pro:n2cocl}.

\smallskip
\noindent
\ref{dim2ii}) 
Assume that there exists an oriented $4$-dimensional subspace $\tilde\gr\subset \gr$ with $\d(\mn')^*\subset \Lambda^2\tilde\mr^*$ and such that for every orthonormal basis $\{\zeta_1,\zeta_2\}$ of $(\mn')^*$, 
the self-dual components $\d\zeta_1^+,\ \d\zeta_2^+\in \Lambda^2_{\sst+}\tilde\gr^*$ of $\d\zeta_1,\ \d\zeta_2\in \Lambda^2\tilde \gr^*$ are orthogonal and have equal norms, say $\rho$.

We define $e_7$ as a unit vector in $\mr$ orthogonal to $\tilde\mr$, $e_5,e_6$ as the metric duals of $\zeta_1,\zeta_2$, and $\alpha_5\coloneqq\d \zeta_1$,  $\alpha_6\coloneqq\d \zeta_2$.
Like before, the transitivity of $\SO(4)$ on pairs of orthogonal forms of fixed length in $\Lambda^2_{\sst+}\tilde \gr^*$ shows that there exists an oriented orthonormal basis 
$\{e_1,\ldots,e_4\}$ of $\tilde\gr$ such that $\alpha_5^+= \frac{\rho}{\sqrt2}(e^{13}-e^{24})$ and $\alpha_6^+ = \frac{\rho}{\sqrt2}(e^{14}+e^{23})$. 
Then, the system \eqref{n2coclosed} holds for $b=0$ and $a=-c=\frac{\rho}{\sqrt2}$. 
Moreover, the structure equations \eqref{str2} hold for $\theta_5=\theta_6=0$ since $\alpha_5,\alpha_6\in \Lambda^2\tilde\mr^*$. 
Therefore, the G$_2$-structure induced by the orthonormal basis $\{e_1,\ldots,e_7\}$ of $\gn$ is purely coclosed by Corollary \ref{pro:n2cocl}.
\end{proof}

In Section \ref{metricn2} we will use the criterion above together with the recent classification of 2-step nilpotent metric Lie algebras of dimension 6 with 2-dimensional derived algebra \cite{DiSc,ReVi}, 
in order to study which of them admit compatible purely coclosed G$_2$-structures.

%%%%%%%%%%%%%%%%%%%%%%%%%%%%%%%%%%%%%%%%%%%%%%%%%%%%%%%%%%%%%%%%%%%%%%%%%%%%%%%%%%%%%%%%%
%%%%%%%%%%%%%%%%%%%%%%%%%%%%%%%%%%%%%%%%%%%%%%%%%%%%%%%%%%%%%%%%%%%%%%%%%%%%%%%%%%%%%%%%%
%																DIM n' = 3
%%%%%%%%%%%%%%%%%%%%%%%%%%%%%%%%%%%%%%%%%%%%%%%%%%%%%%%%%%%%%%%%%%%%%%%%%%%%%%%%%%%%%%%%%
%%%%%%%%%%%%%%%%%%%%%%%%%%%%%%%%%%%%%%%%%%%%%%%%%%%%%%%%%%%%%%%%%%%%%%%%%%%%%%%%%%%%%%%%%

\subsection{Case 3: $\dim(\gn')=3$}\label{subsec:case3} 
To conclude the discussion, we have to investigate the case when $\gn$ is a 7-dimensional 2-step nilpotent Lie algebra with $\dim(\gn')=3$. 
We begin by studying $\G_2$-structures calibrating the derived algebra. 
Then, in Lemma \ref{lcal}, we prove that we can always restrict our study of (purely) coclosed $\G_2$-structures inducing a given metric to those calibrating $\mn'$.

By definition, a 3-dimensional subspace of $\gn$ is {\em calibrated} by $\f$ if and only if there exists an orthonormal basis $\{z_1,z_2,z_3\}$ of it such that $\f(z_1,z_2,z_3)=\pm 1$.  
For more information on the theory of calibrations, we refer the reader to \cite{HaLa}. 

We start by characterizing (purely) coclosed G$_2$-structures calibrating $\gn'$ in terms of adapted bases. 
Consider a 7-dimensional 2-step nilpotent metric Lie algebra $(\gn,g)$  with $\dim (\gn')=3$.
Let $\{e_1,\ldots,e_7\}$ be a $g$-orthonormal basis of $\gn$ such that $\gr = \langle e_1,e_2,e_3,e_4\rangle$ and $\gn'=\langle e_5,e_6,e_7\rangle$.  
The structure equations of $\gn$ are given by 
\[
\begin{dcases}
\d e^k=0,\qquad k=1,2,3,4,\\
\d e^k \eqqcolon \alpha_{k-4},\quad k=5,6,7, 
\end{dcases}
\]
where $\alpha_1,\alpha_2,\alpha_3\in \Lambda^2\gr^*$. 
With respect to the orientation of $\gr$ induced by $e^{1234}$, consider again the basis of $\Lambda^2_{\sst+}\gr^*$:
\[
\sigma_1 = e^{13}-e^{24},\quad \sigma_2 = -(e^{14}+e^{23}),\quad \sigma_3 =e^{12}+e^{34}.
\]

\begin{lemma} \label{lsa}
The $\mathrm{G}_2$-structure $\f$ induced by the basis $\{e_1,\ldots,e_7\}$ is coclosed if and only if 
\begin{equation}\label{sa1}
g(\sigma_i,\alpha_j) = g(\sigma_j,\alpha_i),\quad i,j\in\{1,2,3\},
\end{equation}
and it is purely coclosed if and only if the following additional
condition holds
\begin{equation}\label{sa2}
\sum_{i=1}^3g(\sigma_i,\alpha_i)=0.
\end{equation}
\end{lemma}

\begin{proof}
Consider the G$_2$-structure $\f$ on $\gn$ induced by $\{e_1,\ldots,e_7\}$.
Then, $\f$ is given by \eqref{G2adapted}, $g_\f=g$ and $\gn'$ is calibrated by $\f$. 
Moreover, we can write 
\[
*_\f\f = \frac12\,\sigma_3^2  + \sigma_1 \W e^{67} +\sigma_2 \W e^{75} + \sigma_3 \W e^{56},
\]
so that
\[
\d*_\f\f = (\sigma_2\W\alpha_3 - \sigma_3\W\alpha_2)\W e^5 + (\sigma_3\W\alpha_1 - \sigma_1\W\alpha_3)\W e^6 + (\sigma_1\W \alpha_2- \sigma_2\W\alpha_1)\W e^7.
\]
Thus $\f$ is coclosed if and only if 
\[
\sigma_i\wedge\alpha_j =\sigma_j\wedge\alpha_i ,\quad i,j\in\{1,2,3\},
\]
which, taking into account that $\sigma_i$ are self-dual forms on $\gr$, is equivalent to \eqref{sa1}.

Moreover, since $\f=\sigma_1\wedge e^5+\sigma_2\wedge e^6+\sigma_3\wedge e^7+e^{567}$, we readily compute
\begin{equation}\label{sd}
\f\wedge\d \f=2(\sigma_1\wedge\d e^5+\sigma_2\wedge\d e^6+\sigma_3\wedge\d e^7)\wedge e^{567},
\end{equation}
so the extra condition for $\f$ being purely coclosed is $\sum_{i=1}^3\sigma_i\wedge\alpha_i=0$, which is equivalent to \eqref{sa1} since $\sigma_i$ are self-dual forms on $\gr$.
\end{proof}

We will now show that the whenever $\gn$ has a (purely) coclosed $\G_2$-structure, it also carries a (purely) coclosed $\G_2$-structure inducing the same metric and calibrating $\gn'$.

\begin{lemma} \label{lcal}
Let $(\gn,g)$ be a $7$-dimensional $2$-step nilpotent metric Lie algebra with $\dim (\gn')=3$. 
If there exists a (purely) coclosed $\G_2$-structure on $\gn$ inducing the metric $g$, then there exists a (purely) coclosed $\G_2$-structure on $\gn$ inducing the metric $g$ and calibrating $\gn'$.
\end{lemma}

\begin{proof} Let $\f$ be a $\G_2$-structure on $\gn$ inducing the metric $g$. If $\gn'$ is calibrated by $\f$ there is nothing to show, so we can assume for the rest of the proof that $\gn'$ is not calibrated.

As G$_2$ acts transitively on ordered pairs of orthonormal vectors on $\bR^7$,  there exists a $g$-orthonormal basis $\{e_1,\ldots,e_7\}$ of $\gn$ adapted to $\f$ such that  $e_6,e_7\in\gn'$. 
We denote by $\tilde e_5$ a unit vector in $\gn'$ orthogonal to $e_6$ and $e_7$ (determined up to sign). 
The stabilizer of $e_6$, $e_7$ in G$_2$ also fixes $e_5$, and acts on $\gr$ as SU$(2)$. Using the transitivity of the  action of SU$(2)$ on spheres in $\R^4$, 
one can assume that the $\gr$-component of $\tilde e_5$ is proportional to $e_4$. 
Consequently, there exist $\l,\mu\in\bR$ with $\l^2+\mu^2=1$ such that $\tilde e_5=\lambda e_4+\mu e_5$. We denote $\tilde e_4\coloneqq\mu e_4-\lambda e_5$ and 
$\tilde e_i \coloneqq e_i$ for $i=1,2,3,6,7$. Then, $\gr=\langle \tilde{e}_1,\tilde{e}_2,\tilde{e}_3,\tilde{e}_4\rangle$, the basis $\{\tilde e_1,\ldots,\tilde e_7\}$ is also $g$-orthonormal 
and the G$_2$-structure $\tilde\f$ induced by it calibrates $\gn'=\langle\tilde{e}_5,\tilde{e}_6,\tilde{e}_7\rangle$. 

We will show that if $\f$ is (purely) coclosed, then $\tilde \f$ is (purely) coclosed too. Expressing 
\begin{equation}\label{eqe}e_4=\mu\tilde e_4+\l \tilde e_5,\qquad e_5=-\l \tilde e_4+\mu \tilde e_5
\end{equation} 
and using \eqref{G2adapted}, we have 
\begin{eqnarray*}
*_\f\f	&=&e^{1234}+e^{1256}+e^{3456} +e^{1367}+e^{1457}+e^{2357}-e^{2467}\\
	&=&\mu \tilde e^{1234}+\l \tilde e^{1235}-\l \tilde e^{1246}+\mu \tilde e^{1256}+\tilde e^{3456} +\tilde e^{1367}+\tilde e^{1457}\\
	&&-\l\tilde e^{2347}+\mu \tilde e^{2357}-\mu \tilde e^{2467}-\l \tilde e^{2567}.
\end{eqnarray*}
Denoting by $\alpha_i\coloneqq\d\tilde e^{i+4}$ for $i=1,2,3$ and using that $\d \tilde e^i=0$ for $i=1,2,3,4$, we get 
\begin{eqnarray*}
\d*_\f\f&=&-\l \tilde e^{123}\wedge\alpha_1+\l \tilde e^{124}\wedge\alpha_2+\mu \tilde e^{126}\wedge\alpha_1-\mu \tilde e^{125}\wedge\alpha_2+\tilde e^{346}\wedge\alpha_1-\tilde e^{345}\wedge\alpha_2\\
&& +\tilde e^{137}\wedge\alpha_2-\tilde e^{136}\wedge\alpha_3+\tilde e^{147}\wedge\alpha_1-\tilde e^{145}\wedge\alpha_3 +\l\tilde e^{234}\wedge\alpha_3 +\mu \tilde e^{237}\wedge\alpha_1\\
&&-\mu \tilde e^{235}\wedge\alpha_3-\mu \tilde e^{247}\wedge\alpha_2+\mu \tilde e^{246}\wedge\alpha_3+\l \tilde e^{267}\wedge\alpha_1-\l \tilde e^{257}\wedge\alpha_2+\l \tilde e^{256}\wedge\alpha_3.
\end{eqnarray*}

Suppose now that $\f$ is coclosed. Since $\gn'$ is not calibrated by $\f$, we have $\l\ne 0$.
Taking the interior product with $\tilde e_5$ and $\tilde e_6$ in the previous relation, i.e., $\tilde e_5\lrcorner \tilde e_6\lrcorner \d*_\f\f$, yields $\tilde e^2\wedge\alpha_3=0$. 
Similarly, taking the interior product with $\tilde e_6$ and $\tilde e_7$, or with $\tilde e_5$ and $\tilde e_7$, we get  $\tilde e^2\wedge\alpha_1=\tilde e^2\wedge\alpha_2=0$. 
The above relation thus simplifies to 
\begin{equation}\label{ecnc}
0=\tilde e^{346}\wedge\alpha_1-\tilde e^{345}\wedge\alpha_2 +\tilde e^{137}\wedge\alpha_2-\tilde e^{136}\wedge\alpha_3+\tilde e^{147}\wedge\alpha_1-\tilde e^{145}\wedge\alpha_3.
\end{equation}
Denoting as before 
\[
\sigma_1 = \tilde e^{13}- \tilde e^{24},\quad \sigma_2 = -( \tilde e^{14}+ \tilde e^{23}),\quad \sigma_3 = \tilde e^{12}+ \tilde e^{34}, 
\]
and using again that $e^2\wedge\alpha_i=0$ for $i=1,2,3$, we readily obtain that \eqref{ecnc} is equivalent to 
$g(\sigma_i,\alpha_j) = g(\sigma_j,\alpha_i)$ for all $i,j\in\{1,2,3\}$. By Lemma \ref{lsa}, $\tilde\f$ is thus coclosed.

Finally, we show that $\tilde \f\wedge\d\tilde \f=\f\wedge\d\f$. 
Since $\{\tilde{e}^5,\tilde{e}^6,\tilde{e}^7\}$ spans $\gn'$, \eqref{sd} gives
\begin{equation}\label{sd1}
\tilde \f\wedge\d\tilde \f=2(\sigma_1\wedge\alpha_1+\sigma_2\wedge\alpha_2+\sigma_3\wedge\alpha_3)\wedge \tilde e^{567}.
\end{equation}

The argument above shows that $\d(e^2\wedge\alpha_i)=-e^2\wedge\d\alpha_i=0$, for $i=1,2,3$. 
Using \eqref{eqe}, we obtain $\d e^4=\d(\mu\tilde e^4+\l \tilde e^5)=\l\alpha_1$ and $\d e^5=\d(-\l \tilde e^4+\mu \tilde e^5)=\mu\alpha_1$, so we can write
\begin{eqnarray*}
\f\wedge\d\f	&=&\f\wedge\d( e^{127}+e^{347}+e^{567} + e^{135}-e^{146}-e^{236}-e^{245})\\
			&=&(e^{347}+e^{567} + e^{135}-e^{146})\wedge\d(e^{347}+e^{567} + e^{135}-e^{146})\\
			&=&(e^{347}+e^{567} + e^{135}-e^{146})\wedge(-\l e^{37}\wedge\alpha_1+e^{34}\wedge\alpha_3+\mu e^{67}\wedge\alpha_1\\
			&&-e^{57}\wedge\alpha_2+e^{56}\wedge\alpha_3
				+ \mu e^{13}\wedge\alpha_1+\l e^{16}\wedge\alpha_1-e^{14}\wedge\alpha_2)\\
 			&=&2(\l e^{13467}\wedge\alpha_1+e^{34567}\wedge\alpha_3+\mu e^{13567}\wedge\alpha_1-e^{14567}\wedge\alpha_2).
\end{eqnarray*}
Since $\lambda e^4+\mu e^5=\tilde e^5$ and $e^4\wedge e^5=\tilde e^4\wedge\tilde e^5$, the above relation reads
\begin{eqnarray*}\f\wedge\d\f&=&2(\tilde e^{13567}\wedge\alpha_1-\tilde e^{14567}\alpha_2+\tilde e^{34567}\wedge\alpha_3)\\
&=&2(\tilde e^{13}\wedge\alpha_1-\tilde e^{14}\wedge\alpha_2+\tilde e^{34}\wedge\alpha_3)\wedge \tilde e^{567}\\
&=&2(\sigma_1\wedge\alpha_1+\sigma_2\wedge\alpha_2+\sigma_3\wedge\alpha_3)\wedge \tilde e^{567}.
\end{eqnarray*}
By \eqref{sd1} we thus have $\tilde \f\wedge\d\tilde \f=\f\wedge\d\f$. So, if  $\f$ is purely coclosed, then $\tilde\f$ is purely coclosed too.
\end{proof}

We are now ready to prove the remaining parts of the theorems stated in the Introduction. The next Proposition \ref{p3} gives the proof of Theorem \ref{CCG2THM} for the case $\dim(\gn')=3$, 
and part (\ref{thiii}) of Theorem \ref{th}. 

\begin{proposition}\label{p3}
Let $(\gn,g)$ be a $7$-dimensional $2$-step nilpotent metric Lie algebra with $\dim (\gn')=3$. 
\begin{enumerate}[$1)$]
\item\label{p31}
There exists a coclosed $\G_2$-structure on $\gn$ inducing the metric $g$. 
\item \label{p32}
Furthermore, there exists a purely coclosed $\G_2$-structure on $\gn$ inducing the metric $g$ if and only if for some orientation of the $4$-dimensional space $\gr$, 
and for every orthonormal basis $\{\zeta_1,\zeta_2,\zeta_3\}$ of $(\mn')^*$, the $3\times 3$ Gram matrix $S$ of the self-dual components of their differentials in $\Lambda^2_{\sst+}\gr^*$, 
defined by $S_{ij} \coloneqq g( \d\zeta_i^+,\d\zeta_j^+)$, satisfies $\tr^2(S)=2\tr(S^2)$. 
\end{enumerate}
\end{proposition}
\begin{proof} \ \\ 
\ref{p31}) Consider any $g$-orthonormal bases $\{e_1,e_2,e_3,e_4\}$ and $\{ e_5, e_6, e_7\}$ of $\mr$ and $\mn'$, respectively, and denote by $\{e^5, e^6, e^7\}$ the dual basis of $(\gn')^*$. 
Fix the orientation $e^{1234}$ of $\mr$ and let $\{\sigma_1,\sigma_2,\sigma_3\}$ be the basis \eqref{sigma} of the space self-dual forms $\Lambda_{\sst+}^2\mr^*$. 
Let $M$ be the $3\times 3$ matrix with entries $M_{ij} \coloneqq g(\sigma_i,\d e^{j+4})$. 
We will show that the matrix $M$ can be made symmetric after a change of basis in $\mn'$.

Indeed, it follows from the polar decomposition of $M$ that there exists $P\in\mathrm{O}(3)$ such that $MP$ is symmetric. 
Using the matrix $P,$ we construct a new orthonormal basis $\{\tilde e^5, \tilde e^6,\tilde  e^7\}$ of $(\mn')^*$ as follows: for $j=1,2,3$ we define $ \tilde e^{j+4}=\sum_{k=1}^3P_{kj}  e^{k+4}$, 
and set $ \alpha_j \coloneqq\d  \tilde e^{j+4}$. 
A straightforward computation shows that $g(\sigma_i, \alpha_j)=(MP)_{ij}$. 
Therefore, by Lemma \ref{lsa}, the  $\G_2$-structure induced by the $g$-orthonormal basis $\{ e^1, e^2,e^3, e^4, \tilde e^5,\tilde e^6, \tilde e^7\}$ is coclosed. 
 
\smallskip \noindent
\ref{p32}) Suppose that $\mn$ admits a purely coclosed $\G_2$-structure inducing $g$. Then, by  Lemma \ref{lcal}, there also exists a purely coclosed G$_2$-structure $\f$ calibrating $\mn'$ and inducing $g$. 
Therefore, one can find an orthonormal basis $\{e_1, \ldots, e_7\}$ adapted to $\f$ such that $e_5,e_6,e_7$ span the derived algebra $\gn'$. 
We denote by $\{e^1,\ldots,e^7\}$ the dual basis of $\gn^*$ and we let  $\alpha_i \coloneqq \d e^{i+4}$, for $i=1,2,3$.

Consider the orientation of $\gr$ determined by  $e^{1234}$ and the self-dual forms $\sigma_i\in \Lambda_{\sst+}^2\mr^*$ given by \eqref{sigma}. 
By Lemma \ref{lsa}, the matrix $M$ with entries $M_{ij}\coloneqq g(\sigma_i,\alpha_j)$ is symmetric and trace-free.  
Since $\left\{\frac1{\sqrt2}\sigma_1,\frac1{\sqrt2}\sigma_2,\frac1{\sqrt2}\sigma_3\right\}$ 
forms an orthonormal basis of $\Lambda_{\sst+}^2\mr^*$, the Gram matrix $(S_{ij}) \coloneqq (g( \alpha_i^+,\alpha_j^+))$
is given by $S=\frac12M^2$. The fact that it verifies the required condition $\tr^2(S)=2\tr(S^2)$ is a consequence of the ``only if'' part of Lemma \ref{alg} below, with $P=\mathrm{I}_3$.

Moreover, for any other orthonormal coframe $\{\zeta_1, \zeta_2,\zeta_3\}$ of $(\mn')^*$, 
the Gram matrix $S'$ with entries $g( \d\zeta_i^+,\d\zeta_j^+)$ differs from the matrix $S$ above by conjugation with the $3\times 3$ matrix expressing $\{\zeta_1, \zeta_2,\zeta_3\}$ in terms of $\{e^5,e^6,e^7\}$, 
so the trace condition $\tr^2(S')=2\tr(S'^2)$ still holds.

Conversely, suppose that for some orientation of the $4$-dimensional space $\gr$, and for every 
orthonormal basis $\{\zeta_1,\zeta_2,\zeta_3\}$ of $(\mn')^*$, the Gram matrix $S$ with entries $S_{ij}\coloneqq g( \d\zeta_i^+,\d\zeta_j^+)$ satisfies $\tr^2(S)=2\tr(S^2)$. 

Let $\{e^1, \ldots, e^4\}$ be an oriented $g$-orthonormal basis of $\mr^*$ with respect to the given orientation, and let $\{\zeta_1,\zeta_2,\zeta_3\}$ be a $g$-orthonormal basis of  $(\mn')^*$. 
Consider the basis \eqref{sigma} of the space self-dual forms $\Lambda_{\sst+}^2\mr^*$, and let $M$ be the matrix with entries $M_{ij}\coloneqq g(\sigma_i,\d \zeta_j)$. 
Since the 2-forms $\sigma_i$ are mutually orthogonal and have norm $\sqrt2$, 
the Gram matrix of the vectors $\d \zeta_j^+\in \Lambda^2_{\sst+}\gr^*$, for $j=1,2,3,$ satisfies $S=\frac12M^*M$, where $M^*$ denotes the transpose of $M.$
By hypothesis, $S$ verifies $\tr^2(S)=2\tr(S^2)$ which, by Lemma \ref{alg} below, implies that there exists  
$P\in \mathrm{O}(3)$ such that $MP$ is symmetric and trace-free.

As in the first  part of the proof, we  define a new orthonormal basis $\{e^5,e^6,e^7\}$ of $(\mn')^*$ by setting $e^{j+4}=\sum_{k=1}^3P_{kj}\zeta_k$, for $j=1,2,3$. Denoting $\alpha_j\coloneqq\d e^{j+4}$, 
it is straightforward to check that $g(\sigma_i,\alpha_j)=(MP)_{ij}$, so that the $\G_2$-structure induced by the $g$-orthonormal basis $\{e^1, \ldots, e^7\}$ satisfies $g_\f=g$, 
and is purely coclosed by Lemma \ref{lsa}.
\end{proof}

We now give the algebraic result required in the proof of Proposition \ref{p3}.

\begin{lemma}\label{alg} 
Let $M\in M_3(\mathbb{R})$ and let $S\coloneqq \frac12M^*M,$ where $M^*$ denotes the transpose of $M$. 
Then, there exists $P\in\mathrm{O}(3)$ such that $MP$ is symmetric and trace-free, if and only if $\tr^2(S)=2\tr(S^2)$.
\end{lemma}

\begin{proof} 
Assume that $A\coloneqq MP$ is symmetric and trace-free for some $P\in\mathrm{O}(3)$. Its eigenvalues are $a$, $b$ and $-(a+b)$, for some $a,b\in \mathbb{R}$. 
Then, $S=\frac12M^*M=\frac12PA^2P^*$, whence
\[
\tr(S)=\frac12\tr(A^2)=\frac12(a^2+b^2+(a+b)^2)=a^2+b^2+ab,
\]
and
\[
\tr(S^2)=\frac14\tr(A^4)=\frac14(a^4+b^4+(a+b)^4)=\frac12(a^2+b^2+ab)^2=\frac12\tr^2(S).
\]

Conversely, assume that $\tr^2(S)=2\tr(S^2)$ and let $\frac12a^2,\frac12b^2,\frac12c^2$ denote the eigenvalues of $S$, with $a,b,c\in\mathbb{R}_{\ge 0}$. We have
\begin{eqnarray*}
0	&=&	8\tr(S^2)-4\tr^2(S)=2(a^4+b^4+c^4)-(a^2+b^2+c^2)^2\\
	&=&	a^4+b^4+c^4-2(a^2b^2+b^2c^2+c^2a^2)\\
	&=&	(a+b+c)(a+b-c)(a-b+c)(a-b-c).
\end{eqnarray*}
Up to a permutation, one can thus assume that $c=a+b$. 
On the other hand, the polar decomposition of $M$ gives a matrix $Q\in \mathrm{O}(3)$ such that $B\coloneqq MQ$ is symmetric and positive semi-definite. 
Since $B^2=(MQ)^*MQ=Q^*M^*MQ=2Q^*SQ$, the eigenvalues of $B$ 
are exactly $a,b,c$, and so there exists $R\in\On(3)$ such that 
$B=R\,\diag(a,b,c)R^*$. We can thus write 
\[
BR\,\diag(1,1,-1)R^*=R\,\diag(a,b,-c)R^*.
\]
This shows that $MP$ is symmetric and trace-free for $P\coloneqq QR\,\diag(1,1,-1)R^*\in\On(3)$.
\end{proof}

\smallskip

Propositions \ref{p3}, \ref{dim2} and \ref{dimnn1} complete the proofs of Theorem \ref{CCG2THM} and Theorem \ref{th} in all cases.
Notice that Theorem \ref{CCG2THM} can also be rephrased as follows. 
\begin{theorem}
Let $\gn$ be $7$-dimensional $2$-step nilpotent Lie algebra not isomorphic to $\gn_{7,2,A}$ or $\gn_{7,2,B}$. 
Then, for any metric $g$ on $\gn$ there exists a coclosed $\G_2$-structure $\f\in\Lambda^3_{\sst+}\gn^*$ such that $g_\f=g$.
\end{theorem}

As discussed in Sect.~\ref{G2mfds}, there is a close interplay between coclosed $\G_2$-structures and half-flat SU(3)-structures. In particular, on a 7-dimensional decomposable 2-step nilpotent metric 
Lie algebra $(\gn,g) = (\tilde\mn,g_{\tilde\mn}) \oplus (\bR,g_{\bR})$ there exists a coclosed G$_2$-structure $\f$ such that $g_\f=g_{\tilde\mn} + g_{\bR}$ if and only if 
there is a half-flat SU(3)-structure on $\tilde\mn$ with corresponding metric $g_{\tilde\mn}$. This fact allows us to deduce the following consequence of the previous result. 

\begin{corollary}
Every metric on a  $6$-dimensional $2$-step nilpotent Lie algebra is induced by a half-flat $\SU(3)$-structure.
\end{corollary}

We conclude this section with a remark on calibrations. We have proved in Lemma \ref{lcal} that if there is a coclosed $\G_2$-structure inducing a given metric on a 2-step nilpotent Lie algebra 
with 3-dimensional derived algebra $\mn'$, then there is also one inducing the same metric and calibrating $\mn'$. 
In fact, with the exception of the Lie algebras $\mn$ isomorphic to $\mn_{7,3,A}$, {\em every} coclosed $\G_2$-structure on $\mn$ calibrates $\mn'$, as the next result shows.

\begin{proposition}\label{corcal}
If $\gn$ is a $7$-dimensional $2$-step nilpotent Lie algebra with $\dim(\gn')=3$ and not isomorphic to $\gn_{7,3,A}$, then every coclosed $\rm{G}_2$-structure $\f$ on $\gn$ calibrates $\gn'$.
\end{proposition}

\begin{proof}

Recall that a $\G_2$-structure $\f$ calibrates $\mn'$ if and only if $\f(z_1,z_2,z_3)=\pm1$ for some (and thus every) $g_\f$-orthonormal basis $\{z_1,z_2,z_3\}$ of  $\mn'$. 
It is easy to check that this is equivalent to having $z_1\lrcorner z_2\lrcorner z_3\lrcorner *_\f\f=0$.

Let $\mn$ be a 7-dimensional 2-step nilpotent Lie algebra with $\dim(\mn')=3$ and let $\f$ be a coclosed $\G_2$-structure on it. 
We claim that if $\f$ does not calibrate $\gn'$, then there exits a non-zero covector $\xi\in\mn^*$ such that  that for every $\zeta\in\gn^*$, $\xi\wedge\d \zeta=0$.

Let  $\{z_1,z_2,z_3\}$ be a $g_\f$-orthonormal basis of $\gn'$ and denote by $\{z^1,z^2,z^3\}$ the dual basis of $(\gn')^*$. Since $\gn'\subset\gz$, the Cartan formula shows that 
$z\lrcorner \d\alpha=-\d(z\lrcorner\alpha)$, for every $z\in\gn'$ and $\alpha\in\Lambda^k\gn^*$. We thus have
\begin{equation}\label{zz0}
\d(z_1\lrcorner z_2\lrcorner *_\f\f)=-z_1\lrcorner \d(z_2\lrcorner *_\f\f)=z_1\lrcorner z_2\lrcorner (\d*_\f\f) = 0.
\end{equation}
The 2-form $z_1\lrcorner z_2\lrcorner *_\f\f$ vanishes on $z_1$ and $z_2$, so decomposing $\gn=\gr\oplus\gn'$ it can be written as
\begin{equation}\label{zz}
z_1\lrcorner z_2\lrcorner *_\f\f=z^3\wedge \xi+\gamma,
\end{equation}
with $\xi\in\gr^*$ and $\gamma\in\Lambda^2\gr^*$. Then, $\d\xi=0$ and $\d\gamma=0$, and so \eqref{zz0}--\eqref{zz} give
\[
0=\d (z^3\wedge \xi+\gamma)=\d z^3\wedge \xi. 
\]
Similarly, we obtain  $\xi\wedge\d z^1=\xi\wedge\d z^2=0$. Since the differential vanishes on $\gr=(\gn')^\perp$, this implies $\xi\wedge\d \zeta=0$ for every $\zeta\in\gn^*$.
Moreover, if $\f$ does not calibrate $\mn'$, then $\xi\ne 0$. Indeed,
\[
\xi=z_3\lrcorner(z^3\wedge \xi+\gamma)=z_1\lrcorner z_2\lrcorner z_3\lrcorner *_\f\f\ne 0,
\]
thus proving our claim.

To finish the proof, it is enough to observe that if $\mn$ is not isomorphic to $\gn_{7,3,A}$ then, according to the  classification given in Appendix \ref{2stepnilclass} for $\dim(\mn')=3$, 
there is no non-zero covector $\xi$ whose wedge product with all the differentials of covectors in $\gn^*$ vanishes. 
\end{proof}

In the case $\gn_{7,3,A}$ not covered by this result, one can actually construct examples of coclosed G$_2$-structures not calibrating $\gn'$. 

\begin{example} Consider the 7-dimensional 2-step nilpotent Lie algebra 
\[\mn_{7,3,A}= (0,0,0,0,f^{12},f^{23},f^{24}).\] For any non-zero real numbers $a,b,c$, we define a new coframe $\{e^1, \ldots, e^7\}$ as follows
\[
e^i =f^i, \mbox{ for }i=1,2,3, \mbox{ and }
e^4 = a f^7,\, e^5 = f^4,\, e^6 = -bf^6,\, e^7 = cf^5.
\]
With respect to this coframe, the structure equations become $\d e^4=ae^{25}$, $\d e^6=-b e^{23}$ and $\d e^7=c e^{12}$, and the derived algebra is spanned by  $ e_4,e_6,e_7$. 
It is straightforward to check that the G$_2$-structure $\f$ on $\gn_{7,3,A}$ induced by the basis $\{e_1,\ldots,e_7\}$ is always coclosed and that it is purely coclosed if and only if $a+b+c=0$.  
However, $\gn_{7,3,A}'$ is not calibrated by $\f$, as $\f(e_4,e_6,e_7)=0$. 
\end{example}

%%%%%%%%%%%%%%%%%%%%%%%%%%%%%%%%%%%%%%%%%%%%%%%%%%%%%%%%%%%%%%%%%%%%%%%%%%%%%%%%%%%%%%%%%
%%%%%%%%%%%%%%%%%%%%%%%%%%%%%%%%%%%%%%%%%%%%%%%%%%%%%%%%%%%%%%%%%%%%%%%%%%%%%%%%%%%%%%%%%
%														EXAMPLES PURELY COCLOSED METRICS
%%%%%%%%%%%%%%%%%%%%%%%%%%%%%%%%%%%%%%%%%%%%%%%%%%%%%%%%%%%%%%%%%%%%%%%%%%%%%%%%%%%%%%%%%
%%%%%%%%%%%%%%%%%%%%%%%%%%%%%%%%%%%%%%%%%%%%%%%%%%%%%%%%%%%%%%%%%%%%%%%%%%%%%%%%%%%%%%%%%
\section{Classification results for metrics induced by purely coclosed G$_2$-structures} \label{metricn2}

We now use the characterizations obtained in Section \ref{S5} to study the existence of purely coclosed G$_2$-structures inducing a given metric  
on explicit examples of 2-step nilpotent metric Lie algebras $(\gn,g)$. As before, the discussion will be made according to the dimension of the derived algebra $\gn'$. 

Along this section, we denote the symmetric product of two covectors $f^i,f^j\in \mn^*$ by $f^i\odot f^j \coloneqq \frac12 (f^i\otimes f^j+f^j\otimes f^i)$. 

%%%%%%%%%%%%%%%%%%%%%%%%%%%%%%%%%%%%%%%%%%%%%%%%%%%%%%%%%%%%%%%%%%%%%%%%%%%%%%%%%%%%%%%%%
%%%%%%%%%%%%%%%%%%%%%%%%%%%%%%%%%%%%%%%%%%%%%%%%%%%%%%%%%%%%%%%%%%%%%%%%%%%%%%%%%%%%%%%%%
%																 EX dim n' = 1
%%%%%%%%%%%%%%%%%%%%%%%%%%%%%%%%%%%%%%%%%%%%%%%%%%%%%%%%%%%%%%%%%%%%%%%%%%%%%%%%%%%%%%%%%
%%%%%%%%%%%%%%%%%%%%%%%%%%%%%%%%%%%%%%%%%%%%%%%%%%%%%%%%%%%%%%%%%%%%%%%%%%%%%%%%%%%%%%%%% 

\subsection{Case 1: $\dim(\mn')=1$}\label{subn1} 
We begin by describing the metrics on the Lie algebras $\gh_3\oplus\R^4$, $\gh_5\oplus\R^2$ and $\gh_7$ up to equivalence. 
Recall that two metrics $g,g'$ on a Lie  algebra $\mn$ are equivalent if there exists an automorphism $F$ of $\mn$ such that $F^*g'=g$.

\begin{proposition}\label{pro:classifheis} 
Let $\gn$ be a $7$-dimensional $2$-step nilpotent Lie algebra with $\dim (\gn')=1$.  
\begin{enumerate}[$\bullet$]
\item If $\mn=\mh_3\oplus\R^4=(0,0,0,0,0,0,f^{12})$, then any metric on $\mn$ is equivalent to 
\begin{equation}\label{metheis3}
g= {r^2} f^1\odot f^1+f^2\odot f^2+f^3\odot f^3+f^4\odot f^4+f^5\odot f^5+f^6\odot f^6+f^7\odot f^7,
\end{equation}
for some $r>0$.
\item If $\mn=\mh_5\oplus\R^2=(0,0,0,0,0,0,f^{12}+f^{34})$, then any metric on $\mn$ is equivalent to 
\begin{equation}\label{metheis5}
g = {r^2} f^1\odot f^1+f^2\odot f^2+ {s^2} f^3\odot f^3+f^4\odot f^4+f^5\odot f^5+f^6\odot f^6+f^7\odot f^7,
\end{equation}
for some $0<r\leq s$. 
\item If $\mn=\mh_7=(0,0,0,0,0,0,f^{12}+f^{34}+f^{56})$, then any metric on $\mn$ is equivalent to 
\begin{equation}\label{metheis7}
g = {r^2} f^1\odot f^1+f^2\odot f^2+ {s^2} f^3\odot f^3+f^4\odot f^4+ {t^2} f^5\odot f^5+f^6\odot f^6+f^7\odot f^7,
\end{equation}
for some $0< r\leq s\leq t$.
\end{enumerate}
Moreover, the metrics belonging to the same family are pairwise non-equivalent for different values of the parameters. 
\end{proposition}

\begin{proof}
Let $(\mn,g)$ be a $7$-dimensional $2$-step nilpotent metric Lie algebra with 1-dimensional derived algebra. Consider the orthogonal decomposition $\mn=\mr\oplus\mn'$ 
and fix a generator $e_7\in \mn'$ of unit length, with metric dual $e^7$. 
Let $A$ denote the skew-symmetric endomorphism of $\mr$ determined by $\d e^7$, and let $\{e_1, \ldots, e_6\}$ be an orthonormal basis of $\mr$ such that   
\begin{equation}\label{jzh}
A = 
\begin{psmallmatrix}
0&-a_1&0&0&0&0\\ \noalign{\medskip}
a_1&0&0&0&0&0\\ \noalign{\medskip}
0&0&0&-a_3&0&0\\ \noalign{\medskip}
0&0&a_3&0&0&0\\ \noalign{\medskip}
0&0&0&0&0&-a_5\\ \noalign{\medskip}
0&0&0&0&a_5&0
\end{psmallmatrix},
\end{equation}
for some $ a_1\geq a_3\geq a_5\geq 0$, with $a_1\neq0$. Let $\{e^1,\ldots,e^7\}$ be the dual basis of $\{e_1,\ldots,e_7\}$. 
This implies that $\d e^7=a_1e^{12}+a_3e^{34}+a_5e^{56}$, and we easily see that 
the Lie algebra $\mn$ is isomorphic to $\mh_3\oplus\R^4$ when $a_3=a_5=0$, to $\mh_5\oplus\R^2$ when $a_3\neq 0$ and $a_5=0$, and to $\mh_7$ if all of these coefficients are non-zero.  

Consider a new basis $\{f^1,\ldots,f^7\}$ of $\gn^*$ defined as follows: $f^i \coloneqq e^i$, for $i=2,4,6,7,$ and       
\[
f^i \coloneqq 
\begin{dcases}
a_i e^i,\quad \mbox{ if } a_i\neq0,\\
e^i,\qquad \mbox{ if } a_i = 0,
\end{dcases}
\]
for $i=1,3,5$. 
Then, we have $\d f^7=f^{12}+\varepsilon_3 f^{34}+\varepsilon_5 f^{56}$, where $\varepsilon_3,\varepsilon_5\in \{0,1\}$, $\varepsilon_3\geq\varepsilon_5$, 
and $\varepsilon_k$ vanishes whenever $a_k$ does. We can thus assume that this basis is the one defining the Lie algebra structure. 
It is now immediate to check that the expression of the metric $g$ with respect to the basis $\{f^1, \ldots, f^7\}$ is the one given in the statement of the  proposition, 
where $r,s,t$ coincide with $(a_1)^{-1},(a_3)^{-1},(a_5)^{-1}$, respectively, when $a_1,a_3,a_5$ are non-zero, and they are equal to one otherwise. 

To conclude the proof, we need to show that the metrics belonging to the same family are pairwise non-equivalent for different values of the parameters. 
Consider the basis $\{f_1,\ldots,f_7\}$ of $\mn$ with dual basis $\{f^1,\ldots,f^7\}$. Let $g,g'$ be two metrics on $\gn$ belonging to the same family, and let $F$ be an automorphism of $\mn$ such that $F^*g'=g$. 
Then, $F(f_7)=\pm f_7$ and $F$ preserves the subspace $U\coloneqq\lela f_1, \ldots, f_6\rira$. 
Moreover, the endomorphisms $A, A'$ of $U$ corresponding to $\d f^7$ by means  of $g$ and $g'$, respectively, 
verify $FAF^{-1}=\pm A'$. 
Hence, $A^2$ and $(A')^2$ have the same eigenvalues, which correspond to the parameters of the metric.
\end{proof}

Now, for each 2-step nilpotent metric Lie algebra $(\gn,g)$ with $\dim (\gn')=1$ and $g$ as in Proposition \ref{pro:classifheis}, 
we study the existence of a purely coclosed G$_2$-structure $\f$ inducing $g$.

\begin{proposition}\label{proexheis}
A metric Lie algebra $(\mn,g)$ with $\dim(\mn')=1$ and $g$ as in Proposition \ref{pro:classifheis} admits a purely coclosed $\G_2$-structure $\f$ such that $g_\f=g$  
if and only if one of the following conditions holds:
\begin{enumerate}[a)]
\item $\mn\cong \mh_5\oplus\R^2$ and $g$ is as in \eqref{metheis5} with $r=s$;
\item $\mn\cong \mh_7$ and $g$ is as in \eqref{metheis7} with $\frac1r=\frac1s+\frac1t$.
\end{enumerate}
\end{proposition}
\begin{proof} 
We already know from Corollary \ref{cor:pcch3} that $\mh_3\oplus\R^4$ does not admit any purely coclosed G$_2$-structure.
For the remaining two cases, we use the criterion obtained in Proposition \ref{dimnn1} to prove the assertion.  

Suppose that $\mn$ is isomorphic to $\mh_7$ and that $g$ has the form  \eqref{metheis7} with respect to a suitable basis $\{f^1,\ldots,f^7\}$ of $\gn^*$ such that $\d f^7=f^{12}+f^{34}+f^{56}$.   

Then, the matrix associated  to $\d f^7$ with respect to the basis $\{f_1,\ldots,f_6\}$ is given by 
\[
A = 
\begin{psmallmatrix}
0&-r^{-2}&0&0&0&0\\ \noalign{\medskip}
1&0&0&0&0&0\\ \noalign{\medskip}
0&0&0&-s^{-2}&0&0\\ \noalign{\medskip}
0&0&1&0&0&0\\ \noalign{\medskip}
0&0&0&0&0&-t^{-2}\\ \noalign{\medskip}
0&0&0&0&1&0
\end{psmallmatrix},
\]
and we have
\[
\tr(A^4)-\frac14\tr^2(A^2) = \left(\frac1r+\frac1s+\frac1t\right) \left(\frac1r+\frac1s-\frac1t\right) \left(\frac1r-\frac1s+\frac1t\right) \left(\frac1r-\frac1s-\frac1t\right).  
\] 
Since $0<r\leq s\leq t$, we see that the expression above is zero if and only if $\frac1r=\frac1s+\frac1t$. 

A similar discussion shows that on $\gh_5\oplus\R^2$ there exists a purely coclosed G$_2$-structure inducing the metric given in \eqref{metheis5} if and only if $r=s$.
\end{proof}

This classification allows us to determine whether the nilsoliton metrics on 2-step nilpotent Lie algebras with 1-dimensional derived algebra are induced by purely coclosed $\G_2$-structures.
Indeed, from \cite[Thm.~5.1]{Lau}, the nilsoliton metrics on $\gh_5\oplus\R^2$ correspond to $r=s$ in \eqref{metheis5}, and by  
\cite{FeC}, the nilsoliton metrics on $\gh_7$ correspond to $r=s=t$ in \eqref{metheis7}. 
We thus obtain the following consequence of Proposition \ref{proexheis}. 
\begin{corollary}\label{ricci1}
On $\gh_5\oplus\R^2$ there exist purely coclosed G$_2$-structures inducing a nilsoliton metric, while the nilsoliton metrics on $\gh_7$ are not induced by any purely coclosed G$_2$-structure. 
\end{corollary}

%%%%%%%%%%%%%%%%%%%%%%%%%%%%%%%%%%%%%%%%%%%%%%%%%%%%%%%%%%%%%%%%%%%%%%%%%%%%%%%%%%%%%%%%%
%%%%%%%%%%%%%%%%%%%%%%%%%%%%%%%%%%%%%%%%%%%%%%%%%%%%%%%%%%%%%%%%%%%%%%%%%%%%%%%%%%%%%%%%%
%																 EX dim n' = 2
%%%%%%%%%%%%%%%%%%%%%%%%%%%%%%%%%%%%%%%%%%%%%%%%%%%%%%%%%%%%%%%%%%%%%%%%%%%%%%%%%%%%%%%%%
%%%%%%%%%%%%%%%%%%%%%%%%%%%%%%%%%%%%%%%%%%%%%%%%%%%%%%%%%%%%%%%%%%%%%%%%%%%%%%%%%%%%%%%%%

\subsection{Case 2: $\dim(\mn')=2$}\label{subn2}
Let $(\mn,g)$ be a 7-dimensional 2-step nilpotent metric Lie algebra with $\dim (\gn')=2$, and suppose that there is a unit vector $x\in \ma$ 
so that $\mn$ splits into the orthogonal direct sum $(\mn,g)=(\tilde\mn,g_{\tilde\mn})\oplus (\lela x\rira,g_{\lela x\rira})$ (cf.~Proposition \ref{decomp}). 
In this case,  the $4$-dimensional orthogonal complement  
of $\gn'\oplus \langle x\rangle$ in $\gn$ is in fact the orthogonal complement $\tilde{\mr}$ of $\tilde\mn'=\mn'$ in $\tilde\mn$. 
Moreover,  orthonormal coframes  of $((\gn')^*,g_{\gn'})$ correspond to orthonormal coframes of $((\tilde\gn')^*,g_{\tilde\gn'})$, and their differentials coincide as $2$-forms  in $\Lambda^2\tilde{\mr}^*$. 

Therefore, by Proposition \ref{dim2}~\ref{dim2ii}), determining the $2$-step nilpotent metric Lie algebras $(\gn,g)$ with $\dim(\gn')=2$ admitting purely coclosed G$_2$-structures inducing $g$ 
is equivalent to determining the 6-dimensional $2$-step nilpotent metric Lie algebras $(\tilde\mn,g_{\tilde\mn})$ admitting an orientation of $\tilde{\mr}$ and an orthogonal coframe $\{\zeta_1,\zeta_2\}$ of $((\tilde\gn')^*,g_{\tilde\gn'})$ 
for which the self-dual parts of their differentials are orthogonal and have equal norms in $\Lambda^2_{\sst+}\tilde{\mr}^*$.

Any  6-dimensional 2-step nilpotent Lie algebra is isomorphic to one of $\mh_3^\bC$, $\mh_3\oplus \mh_3$, $\mn_{6,2}$ and $\mn_{5,2}\oplus \R$. 
The metrics on these Lie algebras were classified, up to automorphism, by Di Scala in \cite{DiSc} (for $\mh_3^\bC$)  and by Reggiani and Vittone in \cite{ReVi} (for the remaining cases). 
We recall their classification results here.

\begin{proposition}[\cite{DiSc,ReVi}]\label{pro:classif62}
Let $\gn$ be a $6$-dimensional $2$-step nilpotent Lie algebra with $\dim (\gn')=2$.  
\begin{enumerate}[$\bullet$]
\item If $\mn=\mh_3^\bC=(0,0,0,0,f^{13}-f^{24},f^{14}+f^{23})$, then any metric on $\mn$ is equivalent to 
\begin{equation}
g=f^1\odot f^1 + r f^2\odot f^2 +f^3\odot f^3+ s f^4\odot f^4 + E  f^5\odot f^5+ 2 F f^5\odot f^6+Gf^6\odot f^6,
\end{equation}
for some $ 0<s \leq r\leq 1$ and  $E,F,G\geq0$ with  $EG-F^2>0$. 
\item If $\gn = \gh_3\oplus\gh_3=(0,0,0,0,f^{12},f^{34})$, then any metric on $\mn$ is equivalent to
\begin{equation}
g=\sum_{i=1}^4 f^i\odot f^i + 2a f^1\odot f^3 +2b f^2\odot f^4+ E  f^5\odot f^5+ 2 F f^5\odot f^6+Gf^6\odot f^6,
\end{equation}
for some $ 0\leq a \leq b  <1$ and $E,F,G\geq0$ with $EG-F^2>0$.
\item If $\gn= \gn_{6,2}=(0,0,0,0,f^{12},f^{14}+f^{23})$, then any metric on $\mn$ is equivalent to
\begin{equation}
g=\sum_{i=1}^3 f^i\odot f^i + r f^4\odot f^4 +  E  f^5\odot f^5+ 2 F f^5\odot f^6+Gf^6\odot f^6,
\end{equation}
for some $ 0<r \leq 1$, and $E,F,G\geq0$ with $EG-F^2>0$. 
\item If $\gn = \gn_{5,2}\oplus \R=(0,0,0,0,f^{12},f^{13})$, then any metric on $\mn$ is equivalent to
\begin{equation}
g=\sum_{i=1}^4 f^i\odot f^i +   E  f^5\odot f^5+ Gf^6\odot f^6,
\end{equation} for some $0<E\leq G$.
\end{enumerate}
Moreover, the metrics belonging to the same family are pairwise non-equivalent for different values of the parameters. 
\end{proposition}
 
Thus, up to automorphism, we may assume that $(\tilde\gn,g_{\tilde{\gn}})$ is one of the metric Lie algebras described in the previous proposition, and using part \ref{dim2ii}) of Proposition \ref{dim2}, 
we can characterize the existence of purely coclosed G$_2$-structures inducing the given metric $g=g_{\tilde{\gn}}+x^\flat\odot x^\flat$ 
on $\gn = \tilde{\gn}\oplus\langle x\rangle$ in terms of the parameters appearing in $g_{\tilde{\gn}}$.

\begin{proposition}\label{p55}
The metric Lie algebra $(\gh_3^\bC\oplus\R,g)$ with $g_{\gh_3^\bC}$ as in Proposition \ref{pro:classif62} admits a purely coclosed $\G_2$-structure $\f$ such that $g_\f=g$  
if and only if one of the following set of conditions hold 
\begin{enumerate}[a)]
\item $r=s=1$ and any choice of $E,F,G\geq 0$ with $EG-F^2>0$; 
\item $0<r=s<1$ with $F=0$ and  $G=E\left(\frac{\sqrt{rs}+1}{\sqrt{r}+\sqrt{s}}\right)^2$;
\item $0< s< r \leq 1$ with $F=0$ and either $G=E\left(\frac{\sqrt{rs}+1}{\sqrt{r}+\sqrt{s}}\right)^2$ or $G=E\left(\frac{\sqrt{rs}-1}{\sqrt{r}-\sqrt{s}}\right)^2$.  
\end{enumerate}
\end{proposition}
\begin{proof}
Starting with the basis $\{f^1,\ldots,f^6\}$ of $\tilde{\gn}^*\coloneqq(\gh_3^\bC)^*$ given in Proposition \ref{pro:classif62}, we obtain the following $g$-orthonormal basis: 
\[
\begin{split}
&e^1 = f^1,\quad e^2 = \sqrt{r}\,f^2,\quad e^3 = f^3,\quad e^4 = \sqrt{s}\,f^4,\\ 
&e^5 = \sqrt{E}\,f^5+\frac{F}{\sqrt{E}}\,f^6,\quad e^6 = \sqrt{\frac{EG-F^2}{E}}\,f^6, 
\end{split}
\]
where $\{e^5,e^6\}$ is an orthonormal coframe of $(\tilde{\gn}')^*$ and $\tilde{\gr}^*=\langle e^1,e^2,e^3,e^4\rangle$.  
Consequently, the expressions of $\alpha_k\coloneqq \d e^k$, for $k=5,6$, are the following: 
\begin{equation}\label{eq:betash3c}
\begin{split}
\alpha_5 &= \sqrt{E}\,e^{13} + \frac{F}{\sqrt{Es}}\,e^{14} + \frac{F}{\sqrt{Er}}\,e^{23} -\sqrt{ \frac{{E}}{{rs}}}\,e^{24},\\
\alpha_6 &= \sqrt{\frac{EG-F^2}{Es}}\,e^{14} + \sqrt{\frac{EG-F^2}{Er}}\,e^{23}. 
\end{split}
\end{equation}

With respect to the orientation $e^{1234}$ of $\tilde\mr^*$ and the oriented coframe $\{e^1,e^2,e^3,e^4\}$, the basis of self-dual 2-forms on $\tilde \mr^*$ introduced in \eqref{sigma} is 
\begin{equation}\label{sigmaSD}
\sigma_1^+ = e^{13}-e^{24},\quad \sigma_2^+ = -e^{14}-e^{23},\quad \sigma_3^+ = e^{12}+e^{34}.
\end{equation}
When the opposite orientation of $\tilde\mr^*$ is considered, the basis \eqref{sigma} with respect to the oriented coframe $\{-e^1,e^2,e^3,e^4\}$ reads
\begin{equation}\label{sigmaASD}
\sigma_1^- = -e^{13}-e^{24},\quad \sigma_2^- = e^{14}-e^{23},\quad \sigma_3^- = -e^{12}+e^{34}. 
\end{equation}

For $k=5,6$, we denote by $\alpha_k^{+}$ and $\alpha_k^-$ the self-dual parts of $\alpha_k$ with respect to the orientation $e^{1234}$ of $\tilde \mr^*$ and its opposite, respectively. 
By using \eqref{eq:betash3c}--\eqref{sigmaASD} we obtain
\[
\alpha_5^+ = B_1^+\,\sigma_1^+ + B_2^+\,\sigma_2^+, \quad \alpha_5^- = B_1^-\,\sigma_1^- + B_2^-\,\sigma_2^-,\quad
\alpha_6^+ = A_2^+\,\sigma_2^+, \quad \alpha_6^- = A_2^-\,\sigma_2^-,
\]
where 
\[
\begin{aligned} &
B_1^+ = \frac12 \frac{\sqrt{E}(\sqrt{rs}+1)}{\sqrt{rs}},\quad 
B_2^+ = -\frac12 \frac{F(\sqrt{r}+\sqrt{s})}{\sqrt{Ers}},\quad 
A_2^+ = -\frac12 \frac{(\sqrt{r}+\sqrt{s})\sqrt{EG-F^2}}{\sqrt{Ers}}, &\\ &
B_1^- = -\frac12 \frac{\sqrt{E}({\sqrt{rs}-1})}{\sqrt{rs}},\quad
B_2^- = \frac12 \frac{F(\sqrt{r}-\sqrt{s})}{\sqrt{Ers}},\quad 
A_2^- = \frac12 \frac{(\sqrt{r}-\sqrt{s})\sqrt{EG-F^2}}{\sqrt{Ers}}.&
\end{aligned}
\]

Therefore, $\alpha_5^+$ and $\alpha_6^+$ are orthogonal with the same norm if and only if 
\[
\begin{cases}
A_2^+ B_2^+ = 0,\\
\left|A_2^+\right| = \sqrt{(B_1^+)^2 + (B_2^+)^2}. 
\end{cases} 
\]
As $r\geq s >0$ and $EG-F^2>0$, the first equation holds if and only if $F=0$. Substituting this value in the second equation gives $G=E\left(\frac{\sqrt{rs}+1}{\sqrt{r}+\sqrt{s}}\right)^2$.

\smallskip 

The self-dual parts of $\alpha_5$ and $\alpha_6$ with respect to the orientation $-e^{1234}$ are orthogonal and have the same norm if and only if 
\[
\begin{dcases}
A_2^- B_2^- = 0,\\
\left|A_2^-\right| = \sqrt{(B_1^-)^2 + (B_2^-)^2}. 
\end{dcases} 
\]
The first equation holds if either $F=0$ and $r\ne s$, or $r=s$. In the former case, we see that the second equation holds if and only if  $G=E\left(\frac{\sqrt{rs}-1}{\sqrt{r}-\sqrt{s}}\right)^2$. 
In the latter case, the second equation holds if and only if $r=s=1$. Notice that $r=s=1$ implies $\alpha_5^-=\alpha_6^-=0$. 
\end{proof}

\begin{proposition} \label{p56}
The metric Lie algebra $(\gh_3\oplus\gh_3\oplus\R,g)$ with $g_{\gh_3\oplus\gh_3}$ as in Proposition \ref{pro:classif62} admits a purely coclosed $\G_2$-structure inducing $g$ if and only if 
$(ab \pm \sqrt{(1-a^2)(1-b^2)})^2<1$, $G=E$ and $F = -E \left(ab \pm \sqrt{(1-a^2)(1-b^2)} \right)$. 
In particular, there are no purely coclosed $\G_2$-structures inducing a metric $g$ for which the decomposition $\gh_3\oplus\gh_3\oplus \bR$ is orthogonal.
\end{proposition}

\begin{proof}
Consider the following $g$-orthonormal basis of $\tilde{\gn}^*\coloneqq (\gh_3\oplus\gh_3)^*$ 
\[
\begin{split}
& e^1 = f^1 +a\,f^3,\quad e^2 = f^2 + b\,f^4,\quad e^3 = \sqrt{1-a^2}\,f^3,\quad  e^4 = \sqrt{1-b^2}\,f^4,\\
& e^5 = \sqrt{E}\,f^5 + \frac{F}{\sqrt{E}}\,f^6,\quad e^6 =  \sqrt{\frac{EG-F^2}{E}}\,f^6,
\end{split}
\]
where $\{e^5,e^6\}$ is an orthonormal coframe of $(\tilde{\gn}')^*$ and $\tilde{\gr}^*=\langle e^1,e^2,e^3,e^4\rangle$.  
The expressions of $\alpha_k\coloneqq \d e^k$, for $k=5,6$, are the following: 
\[
\begin{split}
\alpha_5 &= \sqrt{E}\,e^{12} - b \sqrt{\frac{{E}}{{1-b^2}}}\,e^{14} + a\sqrt{\frac{{E}}{{1-a^2}}}\,e^{23} + \frac{abE+F}{\sqrt{E(1-a^2)(1-b^2)}}\, e^{34},\\
\alpha_6 &= \sqrt{\frac{{EG-F^2}}{{E(1-a^2)(1-b^2)}}}\, e^{34}. 
\end{split}
\]

As in the previous proposition, depending on the two possible orientations and the corresponding oriented coframes of $\tilde\gr$, the bases of self-dual forms are given by \eqref{sigmaSD} or \eqref{sigmaASD}. 
From the expressions of $\alpha_5$ and $\alpha_6$, we thus obtain that their self-dual parts in the two cases are
\[
\alpha_5^+ = B_2^+\,\sigma_2^+ + B_3^+\,\sigma_3^+,\quad \alpha_5^- = B_2^-\,\sigma_2^- + B_3^-\,\sigma_3^-,\quad 
\alpha_6^+ = A_3^+\,\sigma_3^+,\quad \alpha_6^- = A_3^-\,\sigma_3^-,
\]
where 
\[
A_3^+ =  \frac12\, \sqrt{\frac{{EG-F^2}}{{E(1-a^2)(1-b^2)}}} = A_3^-,
\]
and
\[
\begin{aligned}
&B_2^+ = \frac12  \frac{\sqrt{E}(b\sqrt{1-a^2}-a\sqrt{1-b^2})}{\sqrt{(1-a^2)(1-b^2)}}, \quad B_3^+ = \frac12  \frac{E(\sqrt{(1-a^2)(1-b^2)}+ab)+F}{\sqrt{E(1-a^2)(1-b^2)}},& \\
&B_2^- = -\frac12  \frac{\sqrt{E}(b\sqrt{1-a^2} +a\sqrt{1-b^2})}{\sqrt{(1-a^2)(1-b^2)}}, \quad B_3^- = -\frac12  \frac{E(\sqrt{(1-a^2)(1-b^2)}-ab)-F}{\sqrt{E(1-a^2)(1-b^2)}}. &
\end{aligned}
\]
Thus, we see that $\alpha_5^+$ and $\alpha_6^+$ are orthogonal and have the same norm if and only if 
\[
\begin{dcases}
A_3^+ B_3^+= 0,\\
\left|A_3^+\right| = \sqrt{(B_2^+)^2+(B_3^+)^2},
\end{dcases} 
\Longleftrightarrow
\begin{dcases}
F+E\left(\sqrt{(1-a^2)(1-b^2)}+ab\right)=0,\\
\sqrt{EG-F^2} = E \left|\left(b\sqrt{1-a^2} - a\sqrt{1-b^2}\right)\right|.
\end{dcases}
\]
Now, the expression of $F$ can be deduced from the first equation. We know that $0\leq a \leq b  <1$. If $b=a$, the right hand side of the second equation would be zero. 
Assuming then $b>a$ and plugging the value of $F$ in the second equation and making some computations gives $G=E$. 
Finally, the condition $EG-F^2>0$ is equivalent to  $(ab+\sqrt{(1-a^2)(1-b^2)})^2<1$. 

\smallskip

If we consider the opposite  orientation of $\tilde\mr^*$, we see that the self-dual parts of $\alpha_5$ and $\alpha_6$ are orthogonal with the same norm if and only if 
\[
\begin{dcases}
A_3^- B_3^-= 0,\\
\left|A_3^-\right| = \sqrt{(B_2^-)^2+(B_3^-)^2},
\end{dcases} 
\Longleftrightarrow
\begin{dcases}
E\left(\sqrt{(1-a^2)(1-b^2)}-ab\right)-F=0,\\
\sqrt{EG-F^2} = E\left(b\sqrt{1-a^2}+a\sqrt{1-b^2}\right),
\end{dcases}
\]
and the thesis follows. 

To conclude the proof, it is sufficient to observe that any metric for which the decomposition $\gh_3\oplus\gh_3\oplus \bR$ is orthogonal must satisfy $a=b=0$. 
This is not possible under either of the constraints $(ab \pm \sqrt{(1-a^2)(1-b^2)})^2<1$. 
\end{proof}

\begin{proposition}\label{p57}
The metric Lie algebra $(\gn_{6,2}\oplus\R,g)$ with $g_{\gn_{6,2}}$ as in Proposition \ref{pro:classif62} admits a purely coclosed $\G_2$-structure $\f$ such that $g_\f=g$  
if and only if $F=0$ and either $G = E\frac{r}{(\sqrt{r}+1)^2}$ and $0< r \leq 1$ or $G = E\frac{r}{(\sqrt{r}-1)^2}$ with $0< r < 1$. 
\end{proposition}
\begin{proof}
Consider the following $g$-orthonormal basis of $\tilde{\gn}^*\coloneqq (\gn_{6,2})^*$ 
\[
\begin{split}
& e^1 = f^1,\quad  e^2 = f^2,\quad e^3 = f^3,\quad e^4 = \sqrt{r}\,f^4,\\
& e^5 = \sqrt{E}\,f^5 + \frac{F}{\sqrt{E}}\,f^6,\quad e^6 =  \sqrt{\frac{EG-F^2}{E}}\,f^6. 
\end{split}
\]
The pair $\{e^5,e^6\}$ is an orthonormal coframe of $(\tilde{\gn}')^*$ and $\tilde{\gr}^*=\langle e^1,e^2,e^3,e^4\rangle$.    
The expressions of $\alpha_k\coloneqq \d e^k$, for $k=5,6$, are the following: 
\[
\alpha_5 = \sqrt{E}\,e^{12} + \frac{F}{\sqrt{Er}}\,e^{14} + \frac{F}{\sqrt{E}}\,e^{23},\quad
\alpha_6 = \sqrt{\frac{EG-F^2}{Er}}\,e^{14} + \sqrt{\frac{EG-F^2}{E}}\,e^{23}. 
\]

The self-dual parts of $\alpha_5$ and $\alpha_6$ with respect to the orientations $e^{1234}$ and $-e^{1234}$ of $\tilde\gr$, 
and the corresponding oriented bases \eqref{sigmaSD} and \eqref{sigmaASD}, are the following 
\[
\alpha_5^+ = B_2^+\,\sigma_2^+ + B_3^+\,\sigma_3^+, \quad \alpha_5^- = B_2^-\,\sigma_2^- + B_3^-\,\sigma_3^-,\quad
\alpha_6^+ = A_2^+\,\sigma_2^+, \quad \alpha_6^- = A_2^-\,\sigma_2^-,
\]
where 
\[
\begin{aligned} &
B_2^+ = -\frac12 \frac{F(\sqrt{r}+1)}{\sqrt{Er}},\quad 
B_3^+ = \frac12 \sqrt{E},\quad 
A_2^+ = -\frac12 \frac{(\sqrt{r}+1)\sqrt{EG-F^2}}{\sqrt{Er}}, &\\ &
B_2^- = -\frac12 \frac{F(\sqrt{r}-1)}{\sqrt{Er}},\quad 
B_3^- = -\frac12\sqrt{E},\quad
A_2^- = -\frac12 \frac{(\sqrt{r}-1)\sqrt{EG-F^2}}{\sqrt{Er}}.&
\end{aligned}
\]
 
Now, $\alpha_5^+$ and $\alpha_6^+$ are orthogonal with the same length if and only if  
\[
\begin{cases}
A_2^+ B_2^+ = 0,\\
\left|A_2^+\right| = \sqrt{(B_2^+)^2+(B_3^+)^2}. 
\end{cases} 
\]
Since $0 < r \leq 1$, the first equation gives $F=0$. From the second equation we then get $G = E\frac{r}{(\sqrt{r}+1)^2}$. 
 
The forms $\alpha_5^-$ and $\alpha_6^-$ are orthogonal with the same length if and only if  
\[
\begin{cases}
A_2^- B_2^- = 0,\\
\left|A_2^-\right| = \sqrt{(B_2^-)^2+(B_3^-)^2}. 
\end{cases} 
\]
The first equation is satisfied if either $F=0$ or $r=1$. 
In the first case, solving the system we obtain $r\neq1$ and $G = E\frac{r}{(\sqrt{r}-1)^2}$. In the second case, we get $E=0$, a contradiction.
\end{proof}

\begin{proposition}\label{p58}
The metric Lie algebra $(\gn_{5,2}\oplus\bR^2,g)$ with $g_{\gn_{5,2}\oplus\bR}$ as in Proposition \ref{pro:classif62} admits a purely coclosed $\G_2$-structure inducing the metric $g$ 
if and only if $G=E$. 
\end{proposition}
\begin{proof}
We choose the following $g$-orthonormal basis of $\tilde{\gn}^* \coloneqq (\gn_{5,2}\oplus\bR)^*$
\[
e^1 = f^1,\quad e^2 = f^2,\quad e^3 = f^3,\quad e^4 = f^4,\quad e^5  =\sqrt{E}\,f^5, \quad  e^6 = \sqrt{G}\,f^6,
\]
where $\{e^5,e^6\}$ is an orthonormal coframe of $(\tilde{\gn}')^*$ and $\tilde{\gr}^*=\langle e^1,e^2,e^3,e^4\rangle$.  Then, 
\[
\alpha_5 = \sqrt{E}\,e^{12},\quad \alpha_6 = \sqrt{G}\,e^{13}.
\]

Depending on the two possible orientations of $\tilde\gr$, we see that the self-dual parts of $\alpha_5$ and $\alpha_6$, with respect to the bases given by \eqref{sigmaSD} and \eqref{sigmaASD}, are 
\[
\alpha_5^+ = \frac{\sqrt{E}}{2}\,\sigma_3^+,\quad \alpha_5^- = -\frac{\sqrt{E}}{2}\,\sigma_3^-,\quad \alpha_6^+ = \frac{\sqrt{G}}{2}\,\sigma_1^+,\quad \alpha_6^- = -\frac{\sqrt{G}}{2}\, \sigma_1^-. 
\]
Both $\alpha_5^+,\alpha_6^+$ and $\alpha_5^-,\alpha_6^-$ are orthogonal. Moreover, they have the same norm if and only if $G=E$. 
\end{proof} 

By \cite[Thm.~3.1]{Wil},  the nilsoliton metrics correspond (up to automorphism and scaling) to the following values of the parameters in Proposition \ref{pro:classif62}: 
\begin{enumerate}[$\bullet$]
\item $\gh_3^{\bC}$: $r=s=E=G=1$, $F=0$;
\item $\gh_3\oplus\gh_3$: $a=b=F=0$, $E=G=1$;
\item $\gn_{6,2}$: $r=E=G=1$, $F=0$;
\item $\gn_{5,2}\oplus\R$: $E=G=1$. 
\end{enumerate}
Consequently, we have the following.
\begin{corollary}\label{ricci2}
On the $2$-step nilpotent Lie algebras $\gh_3^{\bC}\oplus\R$ and $\gn_{5,2}\oplus\R^2$ there exist purely coclosed $\G_2$-structures inducing a nilsoliton metric, while the Lie algebras 
$\gh_3\oplus\gh_3\oplus\R$ and $\gn_{6,2}\oplus\R$ do not admit any such structure. 
\end{corollary}

We conclude this subsection by giving explicit examples of purely coclosed G$_2$-structures on each of the 7-dimensional decomposable 2-step nilpotent Lie algebras with 2-dimensional derived algebra. 
The given structures satisfy the criteria obtained in Propositions \ref{p55}, \ref{p56},  \ref{p57} and \ref{p58}.

\begin{example} 
On the Lie algebras listed below, the $\G_2$-structures induced by the following coframes are purely coclosed:
\begin{enumerate}[$\bullet$]
\item $\gh_3^{\bC}\oplus\bR =  \left(0,0,0,0,f^{13}-f^{24},f^{14}+f^{23},0\right)$ with the coframe
\[
\left\{f^1,~ \sqrt{r}\,f^2,~ f^3,~ \sqrt{s} f^4,~ \sqrt{E}\, f^5,~ \sqrt{E}\frac{\sqrt{rs}+1}{\sqrt{r}+\sqrt{s}}\,f^6,~ f^7\right\},
\]
for any $0< s < r \leq 1$ and $E>0$.
\item $\gh_3\oplus\gh_3\oplus\bR = \left(0,0,0,0,f^{12},f^{34},0\right)$  with the coframe
\[
\left\{f^1 +a\,f^3,~ -\sqrt{1-a^2}\,f^3,~ f^2 + b\,f^4,~ \sqrt{1-b^2}\,f^4,~  \sqrt{E}\,f^5 + \frac{F}{\sqrt{E}}\,f^6,~  \sqrt{\frac{E^2-F^2}{E}}\,f^6,~ f^7\right\},
\]
where $(ab + \sqrt{(1-a^2)(1-b^2)})^2<1$, and $F = -E \left(ab + \sqrt{(1-a^2)(1-b^2)} \right)$. 
\item $\gn_{6,2}\oplus\bR	= \left(0,0,0,0,f^{12},f^{14}+f^{23},0\right)$  with the coframe
\[
\left\{f^1,~ -f^3,~ f^2,~ \sqrt{r}\,f^4,~{\frac{\sqrt{Er}}{\sqrt{r}+1}}\,f^6,-\sqrt{E}\,f^5,f^7 \right\}, 
\]
for any $0<r\leq 1$ and $E>0$.
\item $\gn_{5,2}\oplus\bR^2 = \left(0,0,0,0,f^{12},f^{13},0\right)$  with the coframe
\[
\left\{f^1,~ f^4,~ f^3,~ f^2,~ \sqrt{E}\,f^6,~ \sqrt{E}\,f^5,~f^7 \right\}, 
\]
with $E>0$. 
\end{enumerate}
In each case, the fact that the induced $\G_2$-structure is purely coclosed, follows from Corollary \ref{pro:n2cocl}. 
\end{example}

%%%%%%%%%%%%%%%%%%%%%%%%%%%%%%%%%%%%%%%%%%%%%%%%%%%%%%%%%%%%%%%%%%%%%%%%%%%%%%%%%%%%%%%%%
%%%%%%%%%%%%%%%%%%%%%%%%%%%%%%%%%%%%%%%%%%%%%%%%%%%%%%%%%%%%%%%%%%%%%%%%%%%%%%%%%%%%%%%%%
%																 EX dim n' = 3
%%%%%%%%%%%%%%%%%%%%%%%%%%%%%%%%%%%%%%%%%%%%%%%%%%%%%%%%%%%%%%%%%%%%%%%%%%%%%%%%%%%%%%%%%
%%%%%%%%%%%%%%%%%%%%%%%%%%%%%%%%%%%%%%%%%%%%%%%%%%%%%%%%%%%%%%%%%%%%%%%%%%%%%%%%%%%%%%%%%

\subsection{Case 3: $\dim(\gn')=3$} \label{subn3} 
Currently, no classification of the equivalence classes of metrics on 7-dimensional 2-step nilpotent Lie algebras with 3-dimensional derived algebra is available. 
We will therefore restrict ourselves to constructing, on each such algebra,  a purely coclosed G$_2$-structure, as well as a metric which is not compatible with any purely coclosed G$_2$-structure. 
For the existence part we will use Lemma \ref{lsa}, and for the non-existence part we will apply Proposition \ref{p3}.

Let $\gn$ be a 7-dimensional 2-step nilpotent Lie algebra with 3-dimensional derived algebra $\gn'$. 
We recall the following notation used in Section \ref{subsec:case3}. Assume that $\mathcal{B}=\{e^1,\ldots,e^7\}$ is a coframe of $\gn^*$ such that $e^1, e^2,e^3, e^4$ vanish on $\gn'$. 
We denote by $g$ the metric in which this coframe is orthonormal and by $\f$ the G$_2$-structure induced by $\mathcal{B}$ via \eqref{G2adapted}. 
If $\{e_1,\ldots,e_7\}$ denotes the basis of $\gn$ dual to $\mathcal{B}$, then $\gr = \langle e_1,e_2,e_3,e_4\rangle$ and $\gn'=\langle e_5,e_6,e_7\rangle$.  
In particular, $\gn'$ is calibrated by $\f$ by construction. 

We let $\alpha_i \coloneqq \d e^{i+4}\in \Lambda^2\gr^*$, for $i=1,2,3,$ and we define the $3\times 3$ matrices $M^\mathcal{B}$ and $S^\mathcal{B}_\pm$  
with coefficients $(M^\mathcal{B})_{ij} \coloneqq g(\sigma_i,\alpha_j)$ and $(S^\mathcal{B}_\pm)_{ij} \coloneqq g(\alpha^\pm_i,\alpha^\pm_j)$, where 
\[
\sigma_1 = e^{13}-e^{24},\quad \sigma_2 = -(e^{14}+e^{23}),\quad \sigma_3 =e^{12}+e^{34},
\]
and $\alpha^+_i$ and $\alpha^-_i$ denote the self-dual and anti-self-dual parts of $\alpha_i$ with respect to the metric and orientation of $\gr$ for which $\{e^1,e^2,e^3,e^4\}$ is an 
oriented orthonormal coframe. 
Notice that the self-dual forms with respect to the opposite orientation $-e^{1234}$ of $\gr$ are the anti-self-dual forms with respect to the orientation $e^{1234}$. 

By Lemma \ref{lsa}, $\f$ is purely coclosed if and only if $M^\mathcal{B}$ is symmetric and trace-free. 
Moreover, from Proposition \ref{p3} we know that if 
\begin{equation}\label{in}
\tr^2\left( S^\mathcal{B}_+\right)\ne 2\tr\left((S^\mathcal{B}_+)^2\right),\qquad\hbox{and}\qquad\tr^2\left( S^\mathcal{B}_-\right)\ne 2\tr\left( (S^\mathcal{B}_-)^2\right),
\end{equation}
then $g$ is not compatible with any purely coclosed G$_2$-structure on $\gn$ calibrating $\mn'$. Even more, by Lemma \ref{lcal}, $g$ is not compatible with any purely coclosed G$_2$-structure on $\gn$.

\smallskip

In what follows, for each 7-dimensional 2-step nilpotent Lie algebra $\mn$ with $\dim(\mn')=3$ we will give an example of basis $\mathcal{B}$ such that the metric making it orthonormal is not compatible 
with any purely coclosed G$_2$-structure, and a basis $\mathcal{C}$ inducing a purely coclosed G$_2$-structure. We will explicit the computations in the first case, and sketch the remaining cases. 
\subsubsection{$\gn_{6,3}\oplus\bR	=\left(0,0,0,0,f^{12},f^{13},f^{23}\right)$}

With respect to the coframe $\mathcal{B}=\{f^1,\ldots,f^7\}$, we have 
\[
\alpha_1=\d f^5=f^{12},\quad \alpha_2=\d f^6=f^{13},\quad\alpha_3=\d f^7=f^{23},
\]
whence 
\[
\alpha_1^\pm=\tfrac12(f^{12}\pm f^{34}),\quad \alpha_2^\pm=\tfrac12(f^{13}\mp f^{24}),\quad\alpha_3^\pm=\tfrac12(f^{14}\pm f^{23}).
\]
Consequently  $S^\mathcal{B}_+= S^\mathcal{B}_-=\tfrac12 \mathrm{I}_3$, so the inequalities \eqref{in} hold, showing that the metric for which $\mathcal{B}$ is orthonormal is not compatible 
with any purely coclosed G$_2$-structure.

Consider now the coframe $\mathcal{C}=\{f^1,f^2,f^3,f^4,f^6,2f^7,f^5\}$. With respect to this coframe, we have 
\[
\alpha_1=\d f^6=f^{13},\quad \alpha_2=\d (2f^7)=2f^{23},\quad\alpha_3=\d f^5=f^{12},
\]
showing that
\[
M^\mathcal{C}=
\begin{psmallmatrix}
1&\phantom{-}0&\phantom{-}0\\ \noalign{\medskip}
0&-2&\phantom{-}0\\ \noalign{\medskip}
0&\phantom{-}0&\phantom{-}1
\end{psmallmatrix}
\]
is symmetric and trace-free. Thus the G$_2$-structure induced by $\mathcal{C}$ is purely coclosed.

\subsubsection{$\gn_{7,3,A} = \left(0,0,0,0,f^{12},f^{23},f^{24}\right)$}

With respect to the coframes $\mathcal{B}=\{f^1,\ldots,f^7\}$ and $\mathcal{C}=\{f^1,f^2,f^3,f^4,f^7,f^6,2f^5\}$, we have $S^\mathcal{B}_+= S^\mathcal{B}_-=\tfrac12 \mathrm{I}_3$ and 
\[
M^\mathcal{C}=\begin{psmallmatrix}
-1&\phantom{-}0&\phantom{-}0\\ \noalign{\medskip}
\phantom{-}0&-1&\phantom{-}0\\ \noalign{\medskip}
\phantom{-}0&\phantom{-}0&\phantom{-}2
\end{psmallmatrix}.
\]

Therefore, the metric making $\mathcal{B}$ orthonormal is not compatible with any purely coclosed G$_2$-structure, whereas the G$_2$-structure induced by $\mathcal{C}$ is purely coclosed.

\subsubsection{$\gn_{7,3,B}= \left(0,0,0,0,f^{12},f^{23},f^{34}\right)$}

With respect to the coframes 
\[\mathcal{B}=\left\{f^1,f^2,f^3,f^4,f^5,\frac{1}{\sqrt{2}}\,f^6,f^7\right\}\quad \mbox{ and }\quad \mathcal{C}=\{f^1,f^2,f^3,f^4,f^5-f^7,2f^6,f^5+f^7\},
\] we have 
\[
S^\mathcal{B}_+= \tfrac14 \begin{psmallmatrix}
2&0&2\\ \noalign{\medskip}
0&1&0\\ \noalign{\medskip}
2&0&2
\end{psmallmatrix},\qquad S^\mathcal{B}_-=\tfrac14 \begin{psmallmatrix}
\phantom{-}2&\phantom{-}0&-2\\ \noalign{\medskip}
\phantom{-}0&\phantom{-}1&\phantom{-}0\\ \noalign{\medskip}
-2&\phantom{-}0&\phantom{-}2
\end{psmallmatrix},\qquad M^\mathcal{C}=\begin{psmallmatrix}
0&\phantom{-}0&\phantom{-}0\\ \noalign{\medskip}
0&-2&\phantom{-}0\\ \noalign{\medskip}
0&\phantom{-}0&\phantom{-}2
\end{psmallmatrix},
\]
and one easily checks that 
\[
\tfrac{25}{16}=\tr^2\left( S^\mathcal{B}_\pm\right)\ne 2\tr\left( (S^\mathcal{B}_\pm)^2\right)=\tfrac{17}{8}. 
\]
Thus, the metric making $\mathcal{B}$ orthonormal is not compatible with any purely coclosed G$_2$-structure, whereas the G$_2$-structure induced by $\mathcal{C}$ is purely coclosed. 
Notice that the coframe $\{f^1, \ldots, f^7\}$ is also of the same type as $\mcB$, i.e., the metric making it orthonormal is not compatible with any purely coclosed $\G_2$-structure. 
The interesting property of the coframe $\mcB$ given above is that it is orthonormal with respect to a nilsoliton metric on $\gn_{7,3,B}$ (see Sect.~\ref{nils3}). 

\subsubsection{$\gn_{7,3,B_1}	= \left(0,0,0,0,f^{12}-f^{34},f^{13}+f^{24},f^{14}\right)$}

Consider the coframes $\mathcal{B}=\{f^1,\ldots,f^7\}$ and $\mathcal{C}=\{-f^1,f^2,f^3,f^4,f^6,4f^7,f^5\}$. We compute
\[
S^\mathcal{B}_+= \tfrac12 \begin{psmallmatrix}
0&0&0\\ \noalign{\medskip}
0&0&0\\ \noalign{\medskip}
0&0&1
\end{psmallmatrix},\qquad S^\mathcal{B}_-=\tfrac12 \begin{psmallmatrix}
4&0&0\\ \noalign{\medskip}
0&4&0\\ \noalign{\medskip}
0&0&1
\end{psmallmatrix},\qquad M^\mathcal{C}=\begin{psmallmatrix}
-2&\phantom{-}0&\phantom{-}0\\ \noalign{\medskip}
\phantom{-}0&\phantom{-}4&\phantom{-}0\\ \noalign{\medskip}
\phantom{-}0&\phantom{-}0&-2
\end{psmallmatrix},
\]
and we get
\[
\tfrac 14=\tr^2\left( S^\mathcal{B}_+\right)\ne 2\tr\left( (S^\mathcal{B}_+)^2\right)=\tfrac12,\qquad\hbox{and}\qquad\tfrac{25}4=\tr^2\left( S^\mathcal{B}_-\right)\ne 2\tr\left( (S^\mathcal{B}_-)^2\right)=\tfrac{17}2.
\]
Again, this shows that the metric making $\mathcal{B}$ orthonormal is not compatible with any purely coclosed G$_2$-structure, and that the G$_2$-structure induced by $\mathcal{C}$ is purely coclosed.

\subsubsection{$\gn_{7,3,C} = \left(0,0,0,0,f^{12}+f^{34},f^{23},f^{24}\right)$}\label{73C}

In contrast to the previous cases, the metric making $\{f^1,\dots,f^7\}$ orthonormal turns out to be compatible with a purely coclosed G$_2$-structure. 
To see this, consider the coframes $\mathcal{B}=\{f^1,f^2,f^3,f^4,f^5,f^6,2f^7\}$ 
and $\mathcal{C}=\{f^1,f^2,f^3,f^4,f^7,f^6,f^5\}$. Then, we have
\[
S^\mathcal{B}_+= \tfrac12 \begin{psmallmatrix}
4&0&0\\ \noalign{\medskip}
0&1&0\\ \noalign{\medskip}
0&0&2
\end{psmallmatrix},\qquad S^\mathcal{B}_-=\tfrac12 \begin{psmallmatrix}
0&0&0\\ \noalign{\medskip}
0&1&0\\ \noalign{\medskip}
0&0&2
\end{psmallmatrix},\qquad M^\mathcal{C}=\begin{psmallmatrix}
-1&\phantom{-}0&\phantom{-}0\\ \noalign{\medskip}
\phantom{-}0&-1&\phantom{-}0\\ \noalign{\medskip}
\phantom{-}0&\phantom{-}0&\phantom{-}2
\end{psmallmatrix}.
\]
So
\[
\tfrac{49}4=\tr^2\left( S^\mathcal{B}_+\right)\ne 2\tr\left( (S^\mathcal{B}_+)^2\right)=\tfrac{21}2,\qquad\hbox{and}\qquad \tfrac94=\tr^2\left( S^\mathcal{B}_-\right)\ne 2\tr\left( (S^\mathcal{B}_-)^2\right)=\tfrac92,
\]
showing that the metric making $\mathcal{B}$ orthonormal is not compatible with any purely coclosed G$_2$-structure, whereas the G$_2$-structure induced by $\mathcal{C}$ is purely coclosed.

\subsubsection{$\gn_{7,3,D} = \left(0,0,0,0,f^{12}+f^{34},f^{13},f^{24}\right)$}

With respect to the coframes $\mathcal{B}=\{f^1,\ldots,f^7\}$ and $\mathcal{C}=\{f^1,f^2,f^3,f^4,\tfrac1{\sqrt{2}}(f^7-f^6),\tfrac1{\sqrt{2}}(f^7+f^6),\tfrac1{\sqrt{2}}f^5\}$ we have 
\[
S^\mathcal{B}_+= \tfrac12 \begin{psmallmatrix}
4&\phantom{-}0&\phantom{-}0\\ \noalign{\medskip}
0&\phantom{-}1&-1\\ \noalign{\medskip}
0&-1&\phantom{-}1
\end{psmallmatrix},\qquad S^\mathcal{B}_-=\tfrac12 \begin{psmallmatrix}
0&0&0\\ \noalign{\medskip}
0&1&1\\ \noalign{\medskip}
0&1&1
\end{psmallmatrix},\qquad M^\mathcal{C}=\sqrt2\,\begin{psmallmatrix}
-1&0&0\\ \noalign{\medskip}
\phantom{-}0&0&0\\ \noalign{\medskip}
\phantom{-}0&0&1
\end{psmallmatrix},
\]
and we get
\[
9=\tr^2\left( S^\mathcal{B}_+\right)\ne 2\tr\left( (S^\mathcal{B}_+)^2\right)=10,\qquad\hbox{and}\qquad 1=\tr^2\left( S^\mathcal{B}_-\right)\ne 2\tr\left( (S^\mathcal{B}_-)^2\right)=2. 
\]
Thus, the metric making $\mathcal{B}$ orthonormal is not compatible with any purely coclosed G$_2$-structure, whereas the G$_2$-structure induced by $\mathcal{C}$ is purely coclosed. 
Here again, the reason for the choice of the coframe $\mathcal C$ is that it is orthonormal with respect to a nilsoliton metric on $\gn_{7,3,D}$ (see Sect.~\ref{nils3}). 

\subsubsection{$\gn_{7,3,D_1}	= \left(0,0,0,0,f^{12}-f^{34},f^{13}+f^{24},f^{14}-f^{23}\right)$} 
Also in this case, the standard metric turns out to be compatible with a purely coclosed G$_2$-structure. 
However, in order to construct a metric not compatible with any purely coclosed G$_2$-structure we need to modify the standard one in some $\gr$ direction, 
in contrast to \ref{73C} where the modification was only necessary on $\gn'$. This makes the computation slightly more tricky, since the self-dual and anti-self-dual parts are different from the standard case.

In detail, with respect to the coframe $\mathcal{C}=\{f^1,\ldots,f^7\}$, the matrix $M^\mathcal{C}$ vanishes, 
so it is trivially symmetric and trace-free, showing that the G$_2$-structure induced by $\mathcal{C}$ is purely coclosed.

Consider now the coframe $\mathcal{B}=\{2f^1,f^2,f^3,f^4,f^5,f^6,f^7\}$. The decomposition of 
\[
\alpha_1\coloneqq \d f^5=f^{12}-f^{34},\quad \alpha_2\coloneqq \d f^6=f^{13}+f^{24},\quad\alpha_3\coloneqq \d f^7=f^{14}-f^{23},
\]
into self-dual and anti-self-dual parts with respect to the metric and orientation of $\gr$ determined by the fact that $\{2f^1,f^2,f^3,f^4\}$ is an oriented orthonormal coframe reads 
\begin{eqnarray*}
\alpha_1&=&-\tfrac14(2f^{12}+f^{34})+\tfrac34(2f^{12}-f^{34}),\\
\alpha_2&=&-\tfrac14(2f^{13}-f^{24})+\tfrac34(2f^{13}+f^{24}),\\
\alpha_3&=&-\tfrac14(2f^{14}+f^{23})+\tfrac34(2f^{14}-f^{23}).
\end{eqnarray*}
Since the forms in the brackets are mutually orthogonal and have norm $\sqrt2$, the matrices $S^\mathcal{B}_\pm$ are given by 
$S^\mathcal{B}_+=\tfrac18 \mathrm{I}_3$, $S^\mathcal{B}_-=\tfrac92 \mathrm{I}_3$, 
so the inequalities \eqref{in} hold, showing that the metric making $\mathcal{B}$ orthonormal is not compatible with any purely coclosed G$_2$-structure.

\subsection{Nilsoliton metrics induced by purely coclosed G$_{\mathbf2}$-structures}\label{nils3}
Let $\gn$ be one of the 7-dimensional 2-step nilpotent Lie algebras with $\dim(\gn')=3$, and let $\{f^1,\ldots,f^7\}$ be the corresponding basis of $\gn^*$ given in Appendix \ref{2stepnilclass}.

From \cite{FeC}, we know that, up to automorphism and scaling, a basis of $\mn^*$ which is orthonormal with respect to the nilsoliton metric is given by $\{f^1,\ldots, f^7\}$ for 
the Lie algebras $\mn_{6,3}\oplus\R$, $\mn_{7,3,A}$, $\mn_{7,3,B_1}$,  $\mn_{7,3,C}$ and $\mn_{7,3,D_1}$, by $\left\{f^1,f^2,f^3,f^4,\frac{1}{\sqrt{2}}\,f^5,f^6,f^7\right\}$ for $\mn_{7,3,D}$, 
and by $\left\{f^1,f^2,f^3,f^4,f^5,\frac{1}{\sqrt{2}}\,f^6,f^7\right\}$ for $\mn_{7,3,B}$.

From the discussion in the previous subsections, we can explicitly state which of these nilsoliton metrics are induced by a purely coclosed $\G_2$-structure.

\begin{corollary}\label{ricci3}
On  $\mn_{7,3,C}$, $\mn_{7,3,D}$ and $\mn_{7,3,D_1}$ there exist purely coclosed $\G_2$-structures inducing a nilsoliton metric, while the nilsoliton metrics on 
$\mn_{6,3}\oplus\R$, $\mn_{7,3,A}$, $\mn_{7,3,B}$ and $\mn_{7,3,B_1}$ are not induced by any purely coclosed $\G_2$-structure.  
\end{corollary}

%%%%%%%%%%%%%%%%%%%%%%%%%%%%%%%%%%%%%%%%%%%%%%%%%%%%%%%%%%%%%%%%%%%%%%%%%%%%%%%%%%%%%%%%%
%%%%%%%%%%%%%%%%%%%%%%%%%%%%%%%%%%%%%%%%%%%%%%%%%%%%%%%%%%%%%%%%%%%%%%%%%%%%%%%%%%%%%%%%%
%																 APPENDIX
%%%%%%%%%%%%%%%%%%%%%%%%%%%%%%%%%%%%%%%%%%%%%%%%%%%%%%%%%%%%%%%%%%%%%%%%%%%%%%%%%%%%%%%%%
%%%%%%%%%%%%%%%%%%%%%%%%%%%%%%%%%%%%%%%%%%%%%%%%%%%%%%%%%%%%%%%%%%%%%%%%%%%%%%%%%%%%%%%%%

\appendix

\section{The classification of 7-dimensional 2-step nilpotent Lie algebras}\label{2stepnilclass}

The isomorphism classes of 7-dimensional nilpotent Lie algebras were determined in \cite{Gon}. 
Here, we recall the classification of those that are real and 2-step nilpotent. 

The notation we use is consistent with \cite{Gon}: $\gn_{n,t}$ or $\gn_{n,t,\bullet}$ means that the Lie algebra has dimension $n$ and derived algebra of dimension $t$, 
while different capital letters in the third argument are used to distinguish non-isomorphic Lie algebras whose derived algebras have the same dimension.  
We also denote by $\mh_n$ the Heisenberg Lie algebra of dimension $n$ and by $\gh_3^{\bC}$ the real Lie algebra underlying the complex Heisenberg Lie algebra.

For each Lie algebra $\gn$, the structure equations are written with respect to a basis $\{f^1,\ldots,f^7\}$ of the dual Lie algebra $\gn^*$. 

\begin{enumerate}[$\bullet$]
\item 7-dimensional 2-step nilpotent Lie algebras $\mn$ with $\dim (\mn')=1$:
\begin{eqnarray*}
\gh_3\oplus\bR^4 			&=& \left(0,0,0,0,0,0,f^{12}\right),\\
\gh_5\oplus\bR^2 			&=& \left(0,0,0,0,0,0,f^{12}+f^{34}\right),\\
\mh_7					&=& \left(0,0,0,0,0,0,f^{12}+f^{34}+f^{56}\right).
\end{eqnarray*}
The Heisenberg Lie algebra $\mh_7$ is the only indecomposable one in the above list. \smallskip

\item  7-dimensional 2-step nilpotent Lie algebras $\mn$ with $\dim (\mn')=2$:
\begin{eqnarray*}
\gn_{5,2}\oplus\bR^2 		&=& \left(0,0,0,0,f^{12},f^{13},0\right),\\
\gh_3\oplus\gh_3\oplus\bR 	&=& \left(0,0,0,0,f^{12},f^{34},0\right),\\
\gh_3^{\bC}\oplus\bR 		&=& \left(0,0,0,0,f^{13}-f^{24},f^{14}+f^{23},0\right),\\
\gn_{6,2}\oplus\bR			&=& \left(0,0,0,0,f^{12},f^{14}+f^{23},0\right),\\
\gn_{7,2,A}				&=& \left(0,0,0,0,0,f^{12},f^{14}+f^{35}\right),\\
\gn_{7,2,B}				&=& \left(0,0,0,0,0,f^{12}+f^{34},f^{15}+f^{23}\right).
\end{eqnarray*}
The only indecomposable Lie algebras in the above list are $\gn_{7,2,A}$ and $\gn_{7,2,B}$. \smallskip

\item 7-dimensional 2-step nilpotent Lie algebras $\mn$ with $\dim (\mn')=3$:
\begin{eqnarray*}
\gn_{6,3}\oplus\bR	&=&\left(0,0,0,0,f^{12},f^{13},f^{23}\right),\\
\gn_{7,3,A}		&=& \left(0,0,0,0,f^{12},f^{23},f^{24}\right),\\
\gn_{7,3,B}		&=& \left(0,0,0,0,f^{12},f^{23},f^{34}\right),\\
\gn_{7,3,B_1}		&=& \left(0,0,0,0,f^{12}-f^{34},f^{13}+f^{24},f^{14}\right)\\
\gn_{7,3,C}		&=& \left(0,0,0,0,f^{12}+f^{34},f^{23},f^{24}\right),\\
\gn_{7,3,D}		&=& \left(0,0,0,0,f^{12}+f^{34},f^{13},f^{24}\right),\\
\gn_{7,3,D_1}		&=& \left(0,0,0,0,f^{12}-f^{34},f^{13}+f^{24},f^{14}-f^{23}\right). 
\end{eqnarray*}
The only decomposable Lie algebra in the above list is $\gn_{6,3}\oplus\bR$.
\end{enumerate}

\bigskip\noindent
{\bf Acknowledgements.}
A.R.~was supported by GNSAGA of INdAM and by the project PRIN 2017  ``Real and Complex Manifolds: Topology, Geometry and Holomorphic Dynamics''.  
Part of this work was done during a visit of A.R.~to the Laboratoire de Math\'ematiques d'Orsay of the Universit\'e Paris-Saclay. He is grateful to the LMO for the hospitality.

%%%%%%%%%%%%%%%%%%%%%%%%%%%%%%%%%%%%%%%%%%%%%%%%%%%%%%%%%%%%%%%%%%%%%%%%%%%%%%%%%%%%%%%%%
%%%%%%%%%%%%%%%%%%%%%%%%%%%%%%%%%%%%%%%%%%%%%%%%%%%%%%%%%%%%%%%%%%%%%%%%%%%%%%%%%%%%%%%%%
%																 BIBLIOGRAPHY
%%%%%%%%%%%%%%%%%%%%%%%%%%%%%%%%%%%%%%%%%%%%%%%%%%%%%%%%%%%%%%%%%%%%%%%%%%%%%%%%%%%%%%%%%
%%%%%%%%%%%%%%%%%%%%%%%%%%%%%%%%%%%%%%%%%%%%%%%%%%%%%%%%%%%%%%%%%%%%%%%%%%%%%%%%%%%%%%%%%

\end{document}